\documentclass[10pt,reqno,final]{amsart}

\usepackage{epsfig,amssymb,amsmath,version}
\usepackage{amssymb,version,graphicx,fancybox,mathrsfs,multirow}
\usepackage{epstopdf}
\usepackage{url,hyperref}
\usepackage[notcite,notref]{showkeys}
\usepackage{bm}
\usepackage{subfigure}
\usepackage{color,xcolor}
\usepackage{cases}
\usepackage{mathtools}
\usepackage{extarrows}
\usepackage{bbm}
\usepackage{caption}
\usepackage[bb=ams, cal=cm, scr=boondox, frak=euler]{mathalpha}

\catcode`\@=11 \theoremstyle{plain}
\@addtoreset{equation}{section}   

\@addtoreset{figure}{section}
\renewcommand\thefigure{\thesection.\@arabic\c@figure}
\renewcommand{\thefigure}{\arabic{section}.\arabic{figure}}
\newtheorem{thm}{\bf Theorem}

\newenvironment{theorem}{\begin{thm}} {\end{thm}}
\newtheorem{cor}{\bf Corollary}

\newtheorem{lmm}{\bf Lemma}

\newenvironment{lemma}{\begin{lmm}}{\end{lmm}}

\theoremstyle{example}

\theoremstyle{assumption}
\newtheorem{assumption}{\bf Assumption}[section]

\theoremstyle{remark}
\newtheorem{rem}{\bf Remark}[section]
\theoremstyle{definition}

\numberwithin{table}{section}

\newcommand{\Bu}{{\boldsymbol{u}}}

\newcommand{\Bg}{{\boldsymbol{g}}}
\newcommand{\Be}{{\boldsymbol{e}}}
\newcommand{\Bv}{{\boldsymbol{v}}}
\newcommand{\Bw}{{\boldsymbol{w}}}

\newcommand{\Bx}{{\boldsymbol{x}}}

\renewcommand \wedge \times

\begin{document}
\title{\bf{}}
\title[A novel scheme for SNS equation] {An efficient fully explicit scheme for stochastic Navier-Stokes equations driven by multiplicative noise}
\author[
C. Huang   \; $\&$ \; W.W. Wang\; $\&$ \;  C.J. Xu
]{
	\;\; Can Huang${}^1$ \;\; and \;\; Weiwen Wang${}^2$ \;\; and\;\; Chuanju Xu${}^{1}$
}

\thanks{${}^1$School of Mathematical Sciences, Xiamen university and Fujian Provincial Key Laboratory on Mathematical Modeling \& High Performance Scientific Computing, Xiamen University, Fujian 361005, China. 
	\\
	\indent
	${}^2${School of Mathematics and Statistics, Guangdong University of Technology, Guangdong, Guangzhou, 510520, China.}
}

\begin{abstract}
	This work proposes an efficient, linear, and fully decoupled pressure-correction scheme 
	for the 2D stochastic Navier-Stokes equations with multiplicative noise and Dirichlet boundary condition. Leveraging the auxiliary variable approach, the scheme is fully explicit yet unconditionally stable. At each time step, it only requires solving Poisson-type equations with constant coefficients. To the best of our knowledge, this is the first application of the auxiliary variable method to stochastic Navier-Stokes equations. We provide a detailed strong convergence analysis for the linearized equation under standard assumptions.
\end{abstract}



\maketitle

\section{Introduction}
Let $(\Omega, \mathcal{F}, \{\mathcal{F}_t\}_{t\geq 0}, \mathbf{P})$ be a filtered probability space. In this article,
we investigate a numerical approximation of the following 2D Navier-Stokes (NS) equation perturbed by multiplicative  
noise:
\begin{equation}\label{model:NS}
	\begin{cases}
		\displaystyle{d{\Bu}=\Delta {\Bu}dt-[\Bu\cdot\nabla]\Bu dt-\nabla p dt+G({\Bu})d{ W}(t), \;\; a.s. \  \Bx\in \mathcal{O}=[0,1]^2},\\
		\text{div}\  {\Bu}=0, \;\; \ a.s.,\\
		{\Bu}(0)=\Bu_0,\;\; \ a.s..
	\end{cases}
\end{equation}
Here, $(\Bu,p)$ is the unknown velocity and pressure, $W(t)$ is a perturbed noise, and $G(\Bu)$ is a Nemytskii operator induced by function $g(\Bu)$:
$G(\Bu)\chi(\Bx)=g(\Bu)(\Bx)\chi(\Bx)$ \cite[Lemma 10.24]{LordPS14}.
Both $W(t)$ and $g$ will be specified in the next section. We are particularly interested in the setting of homogeneous Dirichlet boundary condition, due to its related error analysis is more difficult compared to the periodic situation since the equality $\int_\mathcal{O} (\Bu\cdot\nabla)\Bu\cdot \Delta \Bu dx=0$ is no longer available  \cite{BreitP24}. 

It is well-known that ``turbulence in dimensions two and three should be described by the Navier-Stokes equations with a random force
goes back to A. N. Kolmogorov" \cite{KuksinS12}.
In particular, Hairer and Mattingly \cite{HairerM06} proved that the stochastic NS equations on the two-dimensional torus
with finite-dimensional, additive noise have ergodic dynamics and estimated the rate
of convergence to the invariant measure. Inspired by it,
an arsenal of efforts have been devoted to numerical simulation of stochastic NS equation (cf. \cite{BessaihM22, BreitD21, BreitD24, BreitP24, BrzezniakCP13, CarelliP12, Dorsek12, FengV22, HausenblasR19, KukavicaUZ18}). Notably, most of these studies focus on the equation with periodic boundary condition in space, which serves as one of the key motivations for the present work.

Even for the stochastic Navier-Stokes equations with periodic boundary condition, 
the primary challenges in numerical simulation stem from the convection term, 
the noise term, and their interplay. 
A direct consequence of these complexities is the limited regularity of both the velocity and pressure fields. To address these issues, fully implicit methods, which, to some extent, facilitate convergence analysis, are commonly adopted \cite{BessaihM22, BreitD21, BreitD24, BrzezniakCP13, CarelliP12, FengV22, HausenblasR19, KukavicaUZ18}. The standard strategy for ensuring convergence involves localizing the convection term over a space of large probability, thereby obtaining a linear inequality in expectation, which can then be bounded using Gronwall's inequality. This approach ultimately yields a convergence rate in probability.
Notably, under sufficiently smooth noise assumptions, Feng et al. \cite{FengV22} circumvent the localization technique by reversing the order of taking expectations and applying Gronwall's inequality. Their method requires proving exponential integrability of the numerical solutions, rather than relying only on bounds for the usual $p$-th order moments.
Despite significant progress on periodic boundary problems, numerical simulations of \eqref{model:NS} remain relatively scarce (cf. \cite{Breckner00, BreitP23, Doghman24, LiXZ25, OndrejatPW23}). In these studies, the convection term is treated either fully implicitly or in a linearized fashion.

As can be readily observed, a direct consequence of fully implicit methods 
is the necessity of solving a nonlinear system at each time step. 
The existence and uniqueness of solutions to such nonlinear systems,
especially under stochastic perturbations, remain unclear, and identifying the correct solution poses a significant challenge. In this context, linearized methods emerge as a natural alternative. These methods require solving a coupled elliptic system with variable coefficients instead \cite{BessaihM22, BrzezniakCP13, CarelliP12}, and their convergence rates in probability are typically established by comparison with the fully implicit schemes mentioned above \cite{BessaihM22, CarelliP12}. The simplest alternative, however, lies in explicit methods, which reduce each time step to solving a stochastic Stokes system --- though this comes at the cost of a stringent CFL condition.

To balance the strengths and limitations of the three types of methods discussed above, 
we propose a novel pressure-correction scheme for \eqref{model:NS} based on stochastic auxiliary variables. Inspired by \cite{LiSL21}, our scheme treats the convection term in a fully explicit manner using an auxiliary variable while preserving unconditional stability (Lemma \ref{lem:sta}). Combined with a pressure-correction step, the resulting method is linear and decoupled. It is noteworthy that our strategy for managing the stochastic term $G(\Bu) dW(t)$ 
in \eqref{model:NS} involves introducing a second auxiliary variable to counterbalance its random effects. 
As will be shown, both auxiliary variables evolve as stochastic processes with constant mean and small variance.

Owing to the combined use of two auxiliary variables and the pressure-correction technique, 
the proposed scheme exhibits the following notable features:

$\bullet$ The method requires solving only a sequence of Poisson-type equations with constant coefficients, along with a positive-definite $2 \times 2$ linear algebraic system at each time step, regardless of boundary condition. This leads to a substantial improvement in computational efficiency compared to the fully implicit or semi-linear schemes discussed earlier.

$\bullet$ Despite treating the convection term explicitly, the scheme maintains unconditional stability in the $2m$-th moment sense for any $m \geq 1$.

$\bullet$ The auxiliary variable introduced for the stochastic term demonstrates an automatic and adaptive mean-reverting property, which contributes to the robustness of the method.

The remainder of this paper is structured as follows. Section 2 presents some preliminaries, including notations, properties of the convection term, and the assumptions used throughout the paper. In particular, we establish the exponential integrability of the velocity $\Bu$. Section 3 focuses on the auxiliary variable-based temporal semi-discretization and provides an analysis of its unconditional stability. In Section 4, we derive a rigorous strong convergence rate for the proposed scheme, based on the linearized equation \eqref{model:NS1}. Section 5 presents numerical results for the stochastic Navier-Stokes equations to validate the main theoretical results, and Section 6 is the concluding remarks. 

\section{Prerequisite results and Assumptions}

\subsection{Notations and spaces}
Let $(\mathbf{L}^2(\mathcal{O}), \|\cdot\|), (\mathbf{H}^k(\mathcal{O}), \|\cdot\|_k), (\mathbf{H}_0^k(\mathcal{O}), \|\cdot\|_k)$ be standard vector Sobolev space on $\mathbb{R}^d$.  Define the following two frequently used spaces associated with the model. 
$$
\mathbf{H}=\{\Bv\in \mathbf{L}^2(\mathcal{O}): \nabla\cdot \Bv=0\}, 
\;\; \mathbf{V}=\{\Bv\in \mathbf{H}^1(\mathcal{O}): \nabla\cdot \Bv=0\}. 
$$ 
Let $P_\mathbf{H}$ denote the Helmholtz-Leray projector on $\mathbf{H}$. It is well-known that $P_\mathbf{H}$ is stable in both $\mathbf{L}^2$ and $\mathbf{H}^1$ \cite[Page 251]{BoyerF12}. 

Denote the trilinear form by
$
b(\Bu,\Bv,\Bw)=(\Bu\cdot\nabla\Bv, \Bw).
$
It enjoys the following skew-symmetric and orthogonal properties:
\begin{align}
	&b(\Bu,\Bv, \Bw)=-b(\Bu,\Bw,\Bv),\ \ \ \Bu\in\mathbf{H}, \Bv,\Bw\in\mathbf{H}_0^1;\notag\\
	&b(\Bu,\Bv,\Bv)=0, \ \ \ \qquad\qquad \ \  \Bu\in\mathbf{H}, \ \Bv\in \mathbf{H}_{0}^1.
\end{align}
Furthermore, $b(\Bu,\Bv,\Bw)$ satisfies the following inequalities:
\begin{equation}\label{eq:bddb}
	b(\Bu,\Bv,\Bw)\leq 
	\begin{cases}
		c\|\Bu\|_2 \|\Bv\| \|\Bw\|_1,  \ \ \ \forall \Bu\in \mathbf{H}^2\cap\mathbf{H}, \ \Bv\in\mathbf{H}, \ \Bw\in \mathbf{H}_{0}^1; \\
		c\|\Bu\|_2 \|\Bv\|_1 \|\Bw\|,  \ \ \ \forall \Bu\in \mathbf{H}^2\cap\mathbf{H}, \ \Bv\in\mathbf{H}, \ \Bw\in \mathbf{H}_{0}^1; \\
		c\|\Bu\|_1 \|\Bv\|_2 \|\Bw\|,  \ \ \ \forall \Bv\in \mathbf{H}^2\cap\mathbf{H}, \ \Bu\in\mathbf{H}, \ \Bw\in \mathbf{H}_{0}^1; \\
		c\|\Bu\| \|\Bv\|_2 \|\Bw\|_1,  \ \ \ \forall \Bv\in \mathbf{H}^2\cap\mathbf{H}, \ \Bu\in\mathbf{H}, \ \Bw\in \mathbf{H}_{0}^1.
	\end{cases}
\end{equation}
We emphasize that in our numerical analysis, the property $b(\Bv,\Bv, A\Bv)=0$ where $A=-P_\mathbf{H}\Delta$ \cite[Chapter 6]{DaPrato12} will not be used. It is maintained under the restriction of the periodic boundary condition. 

Let $U$ and $V$ be two  Hilbert spaces.  We denote the norm in $L^p(\Omega, \mathcal{F}, \mathbf{P}; U)$ by $\|\cdot\|_{L^p(\Omega; U)}$, that is,
$$
\|Y\|_{L^p(\Omega;U)}=\big(\mathbf{E}\big[\|Y\|_U^p\big]\big)^{\frac{1}{p}},\; \; Y\in L^p(\Omega, \mathcal{F}, \mathbf{P}; U).
$$
Denote by $\mathcal{L}_1(U,V)$ the nuclear operator space from $U$ to $V$ and for $T\in \mathcal{L}_1(U,V)$, its norm is given by
$$
\|T\|_{\mathcal{L}_1}=\sum\limits_{i=1}^\infty |(Te_i,e_i)_U|\ \;\text{and}\;\; Tr(T)=\sum\limits_{i=1}^\infty (Te_i,e_i)_U
$$
for any orthonormal basis $\{e_i\}$ of $U$. Let $\mathcal{L}_2(U,V)$ be the Hilbert-Schmidt space such that for any $T\in \mathcal{L}_2(U,V)$
$$
\|T\|_{\mathcal{L}_2(U,V)}=\|T\|_{HS}=\bigg(\sum\limits_{i=1}^\infty\|Te_i\|^2\bigg)^{1/2}<\infty.
$$

Following \cite{Kruse14}, we introduce the space $U_0:=Q^{\frac{1}{2}}(U)$ together with its inner product
$$
(u_0, v_0)_{U_0}=(Q^{-\frac{1}{2}}u_0, Q^{-\frac{1}{2}}v_0)_U, \ \ \ \forall u_0, v_0\in U_0.
$$
Furthermore, we  introduce  the notations $\mathcal{L}_2^0:=\mathcal{L}_2(U_0, H)$ with its norm
$$
\|T\|_{\mathcal{L}_2^0}=\|TQ^{\frac{1}{2}}\|_{\mathcal{L}_2(U,H)},\ \ \ \ \forall T\in \mathcal{L}(U,H).
$$

\subsection{Q-Wiener process} For a Hilbert space $\mathbf{K}$, define a $\mathbf{K}$-valued wiener process by 
$$
W(t)=\sum\limits_{j=1}^\infty \sqrt{q_j}e_j\beta_j(t),
$$
where $\beta_j(t)$ are standard Brownian motion, and $(q_j, e_j)_{j=1}^\infty$ are eigenpairs of a given symmetric and positive-definite operator $Q$.   
In this work, we shall restrict ourselves to the so-called trace class noise, that is, $\text{tr}(Q)=\sum\limits_{j=1}^\infty q_j<\infty$. 

We shall frequently use the following standard Burkholder-Davis-Gundy (BDG) inequality  and the BDG inequality 
for a sequence of $L^2$-valued discrete martingales.
\begin{lemma}\label{lem:BDG} \cite{Kruse14}
	For any $p\geq 2, 0\leq \tau_1\leq \tau_2\leq T$, and for any predictable stochastic process $\Phi:[0,T]\times\Omega\to \mathcal{L}_2^0$, which satisfies
	$$
	\mathbf{E}\left(\int_{\tau_1}^{\tau_2}\|\Phi(t)\|^2_{\mathcal{L}_2^0}ds\right)^{\frac{p}{2}}<\infty.
	$$
	Then,
	\begin{align*}
		\mathbf{E} \left\|\int_{\tau_1}^{\tau_2}\Phi(t)dt\right\|^p\leq C_p\mathbf{E}\left(\int_{\tau_1}^{\tau_2}\|\Phi(t)\|^2_{\mathcal{L}_2^0}ds\right)^{\frac{p}{2}},
	\end{align*}
	where $C_2=1$ and
	$$C_p=\left(\frac{p}{2}(p-1)\right)^{\frac{p}{2}}\left(\frac{p}{p-1}\right)^{p(\frac{p}{2}-1)}, \  p>2.$$
\end{lemma}

\begin{lemma}\label{lem:dBDG} \cite{HutzenhalerJ11,LiuQ21}
	Let $p\geq 2$ and $\{Z_m\}$ be a sequence of $L^2$-valued random variables with bounded $p$-moments such that 
	$\mathbf{E}[Z_{m+1}|Z_0, \cdots,  Z_m] = 0$ for all $1\leq m\leq N-1$. Then there exists a constant $C=C(p)$ such that
	\begin{align}
		\bigg(\mathbf{E}\bigg\|\sum\limits_{i=0}^m Z_i\bigg\|^p\bigg)^{\frac{1}{p}}\leq C\bigg(\sum\limits_{i=0}^m(\mathbf{E}\|Z_i\|^p)^{\frac{2}{p}}\bigg)^{\frac{1}{2}}.
	\end{align}
\end{lemma}

\subsection{Essential assumptions}
For the sake of convergence analysis and algorithm development, some assumptions are placed in order:
\begin{assumption}\label{assump:1}
	The function $g(\Bu)$ is uniformly bounded and Lipschitz continuous, i.e., there exists a constant $C_g>0$ such that
	\begin{align}
		&\|g(\Bu)\|_\infty \leq C_g,\ \ \; 
		\|g(\Bu)-g(\Bv)\| \leq C_g\|\Bu-\Bv\|,\notag\\
		& \|(-\Delta)^{\frac{1}{{2}}}g(\Bu)Q^{\frac{1}{2}}\|_{HS}\leq C_g,\ \ \ \Bu, \Bv\in \mathbf{H}. \label{assump1eq1}
	\end{align}
\end{assumption}
\begin{rem}
	Clearly, if $g(\Bu)=1$, \eqref{assump1eq1} is reduced to $\|(-\Delta)^{\frac{1}{{2}}}Q^{\frac{1}{2}}\|_{HS}\leq C_g$. This type of assumption is widely used in the literature (cf. \cite{BrehierCH19, KovacsLL11, KovacsLL15, LiMS25, QiW20}).
\end{rem}
\begin{rem} 
	Since $G(\Bu)$ is a Nemytskii operator induced by $g$ and $W(t)$ is of trace-class, one can easily have
	\begin{align}\label{eq:rem1}
		\|\Bg(\Bu)\|_{\mathcal{L}_2^0}^2&=\|g(\Bu)Q^{\frac{1}{2}}\|^2_{\mathcal{L}_2}=\sum\limits_{j=1}^\infty (g(\Bu)Q^{\frac{1}{2}}, e_j)^2
		{ \leq}\|g(\Bu)\|_\infty^2\sum\limits_{j=1}^\infty q_j \|e_j\|^2\leq C_g^2Tr(Q). 
	\end{align}
\end{rem}

\begin{assumption}\label{assump:2} 
	The initial condition 
	$
	\Bu_0\in \mathbf{V}. 
	$
\end{assumption}

Under the Assumptions \ref{assump:1} and \ref{assump:2}, there exists a strong solution $\Bu$ of \eqref{model:NS} such that (cf. \cite{BreitP24,Gourcy17})
\begin{align*}
	(\Bu(t),\boldsymbol{\phi})=(\Bu_0, \boldsymbol{\phi})-\int_0^t (\nabla \Bu(s), \nabla \boldsymbol{\phi})ds-\int_0^t ([\Bu(s)\cdot\nabla]\Bu(s), \boldsymbol{\phi})ds+\int_0^t (g(\Bu(s))dW(s), \boldsymbol{\phi}), \ \boldsymbol{\phi}\in \mathbf{V}.
\end{align*}

The pressure $p$ plays a crucial role in both our algorithm development and error analysis, we need to impose an assumption on it. 
\begin{assumption}\label{assump:3}
	For $m\geq 2$, 
	\begin{align}
		\mathbf{E}\|\nabla p(t)-\nabla p(s)\|^{\frac{m}{2}}\leq C|t-s|^\frac{\theta m}{4},  \forall \ 0\leq\theta<\frac{1}{2}.
	\end{align}
\end{assumption}

\begin{rem}
	This assumption appears to be a necessary trade-off for \eqref{model:NS} with Dirichlet boundary condition, and it is introduced solely for the purpose of error analysis. For the stochastic Navier-Stokes equations with additive noise under periodic boundary condition, the following continuity condition can be established (cf.\cite{HausenblasR19}): 
	\begin{align}
		\mathbf{E}\|p(t)-p(s)\|^{\frac{m}{2}}\leq C|t-s|^\frac{\theta m}{4},  \forall \ 0\leq\theta<\frac{1}{2}.
	\end{align}
\end{rem}
\subsection{Exponential integrability of velocity}

We take a Leray projection $P_\mathbf{H}$ on both sides of \eqref{model:NS} and still denote the main equation as
$$
d{\Bu}=\Delta {\Bu}dt-[\Bu\cdot\nabla]\Bu dt+g({\Bu})d{ W}(t), \;\; a.s. \  x\in \mathcal{O}=[0,1]^2.
$$

The following lemma will be used to obtain exponential integratibility of $\Bu$. 
\begin{lemma}\cite{CuiHL17} \label{lem:staCHL}
	Let $H$ be a Hilbert space, and let $X$ be an adapted $H$-valued stochastic process with continuous sample paths, satisfying 
	$X_t=X_0+\int_0^t \mu(X_r)dr+\int_0^t \sigma(X_r)dW_r$, for all $t\in [0,T]$, where almost surely 
	$\int_0^T \|\mu(X_t)\|+\|\sigma(X_t)\|^2dt<\infty$. Assume that there exist two functionals $\bar{V}$ and $V\in C^2(H; \mathcal{R})$  and an $\mathscr{F}_0$-measurable random variable $\alpha$, such that for almost sure $t\in [0,T]$, 
	$$
	DV(X_t)\mu(X_t)+\frac{1}{2}tr[D^2V(X_t)\sigma(X_t)\sigma^*(X_t)]+\frac{\|\sigma^*(X_t)DV(X_t)\|^2}{2e^{\alpha t}}+\bar{V}(X_t)\leq \alpha V(X_t).
	$$
	Then,
	$$
	\sup\limits_{t\in[0,T]}\mathbf{E} \bigg[\exp\bigg(\frac{V(X_t)}{e^{\alpha t}}+\int_0^t \frac{\bar{V}(X_r)}{e^{\alpha r}}dr\bigg)\bigg]\leq\mathbf{E}V(X_0).
	$$
\end{lemma}

\begin{lemma}\label{lem:eiu}
	Under Assumptions \ref{assump:1} and \ref{assump:2}, the true solution $\Bu$ of \eqref{model:NS} has the following exponential integral stability: for any $\rho\geq 0$
	\begin{align}
		\mathbf{E}{ \bigg[}\exp\bigg(\rho\|\Bu(t)\|^2+\rho\int_0^t \|\nabla \Bu(s)\|^2ds\bigg){ \bigg]}<\infty. 
	\end{align}
\end{lemma}
\begin{proof}
	We largely follow the proof of \cite[Proposition 4.3]{BrehierCH19}. 
	Since in our case $\mu(\Bx)=\Delta \Bx+(\Bx\cdot\nabla)\Bx, \sigma(\Bx)=g(\Bx)$, we choose  $V(\Bx)=\rho\|\Bx\|^2$ and have
	\begin{align}
		&{ \langle DV(\Bx),\mu(\Bx) \rangle}+\frac{1}{2} \text{tr}[g(\Bx)g^*(\Bx)D^2V(\Bx)]+\frac{1}{2} \|g^*(\Bx)DV(\Bx)\|^2 \notag\\
		&\leq -2\rho\|\nabla \Bx\|^2+\rho C_g^2 Tr(Q)+2\rho^2C_g^2\|\Bx\|^2.
	\end{align} 
	Hence, we take $\alpha\geq 2\rho C_g^2$, and define
	$$
	\bar{V}(\Bx)=2\rho \|\nabla \Bx\|^2-\rho C_g^2Tr(Q).
	$$
	Applying Lemma \ref{lem:staCHL}, we can obtain
	\begin{align}
		\mathbf{E} \bigg[\exp\bigg(e^{-\alpha t}\rho\|\Bu(t)\|^2+\rho\int_0^t e^{-\alpha s}\|\nabla \Bu(s)\|^2ds \bigg) \bigg]\leq e^{\rho \|\Bu_0\|^2},
	\end{align}
	which is the desired result. 
\end{proof}
\begin{rem}
	Similar exponential integrability results for 2D stochastic Navier-Stokes equations with additive trace-class noise are provided in \cite{Debussche13, Gourcy17}.
\end{rem}
\begin{cor}\label{cor:up}
	Under Assumptions \ref{assump:1} and \ref{assump:2}, the true solution $\Bu$ of \eqref{model:NS} has high-order bounded moments, i.e. for any $m\geq 1$, 
	\begin{align}
		\mathbf{E}\|\Bu(t)\|^m<\infty, \ \ t\in [0,T].
	\end{align} 
\end{cor}
\begin{proof}
	It is a direct consequence of Lemma \ref{lem:eiu}. 
\end{proof}

Moreover, the velocity has the following H\"{o}lder continuity condition: 
\begin{lemma}
	For $p\geq 2$, 
	\begin{align}\label{eq:conu}
		&\mathbf{E}\|\Bu(t)-\Bu(s)\|_{L^4}^p\leq C|t-s|^{\theta p},\ \,\notag\\
		&\mathbf{E}\|\Bu(t)-\Bu(s)\|_{\mathbf{V}}^p\leq C|t-s|^{\frac{\theta p}{2}}, \ \ \forall\  0<\theta<\frac{1}{2}.
	\end{align}
\end{lemma}
\begin{proof}
	The proofs are the same as those in \cite{CarelliP12} by noting that Lemma 2.2 of Giga \cite{Giga85} is valid for  Dirichlet boundary condition. 
\end{proof}

\section{Semi-discretization and its stability result}

In this section, we will present our scheme based on auxiliary variable approach for \eqref{model:NS}. As will be demonstrated, the scheme is fully explicit in time, while still maintaining unconditional stability, in contrast to its fully implicit counterpart.

\subsection{Algorithm development}
To start with, we rewrite the equation as
\begin{equation}\label{model:NS2}
	\begin{cases}
		\displaystyle{d{\Bu}=\Delta {\Bu}dt-\xi[\Bu\cdot\nabla]\Bu dt-\nabla p dt+\eta g({\Bu})d{ W}(t), \;\; a.s. \  x\in \mathcal{O}=[0,1]^2,}\\
		\text{div}\  {\Bu}=0, \;\; \ a.s. ,\\
		{\Bu}(0)=\Bu_0,\;\; \ a.s. , \\
		d\xi=([\Bu\cdot\nabla]\Bu, \Bu)dt,   \ \xi(0)=1,\\
		d\eta=-(g(\Bu)dW, \Bu-\Bu),\ \ \eta(0)=1.
	\end{cases}
\end{equation}
It is seen that we introduce two auxiliary variables, i.e., $\xi(t)$ and $\eta(t)$ to handle the nonlinear term $(\Bu\cdot\nabla)\Bu dt$ and the stochastic term $g(\Bu)dW(t)$ respectively, 
and
they are theoretically identical to 1. Hence, \eqref{model:NS2} is equivalent to \eqref{model:NS}. 

The pressure-correction scheme based on the auxiliary variables 
can be constructed as follows: 
{find} $(\tilde{\Bu}^{n+1}, \Bu^{n+1}, p^{n+1},$
$ \xi^{n+1}, \eta^{n+1})$ by solving 
\begin{subnumcases}
	{}\tilde{\Bu}^{n+1}-\Bu^n=\tau\Delta \tilde{\Bu}^{n+1}-\tau\xi^{n+1} [\Bu^n\cdot\nabla]\Bu^{n}-\tau\nabla p^{n}+\eta^{n+1}g(\Bu^n)\delta W^n, \label{eq:SAV1}\\
	\qquad\qquad\qquad\qquad\qquad\qquad\qquad\qquad\qquad\qquad\qquad\qquad\tilde{\Bu}^{n+1}|_{\partial\mathcal{O}}=0, \\
	\Bu^{n+1}-\tilde{\Bu}^{n+1}+\tau(\nabla p^{n+1}-\nabla p^n)=0, \label{eq:SAV2} \\
	\nabla\cdot \Bu^{n+1}=0, \; \Bu^{n+1}\cdot {\bf n} |_{\partial\mathcal{O}}=0, \label{eq:SAV21}\\
	\xi^{n+1}-\xi^n={\tau([\Bu^n\cdot\nabla]\Bu^{n}, \tilde{\Bu}^{n+1})}, \label{eq:SAV3}\\
	\eta^{n+1}-\eta^n=-{(g(\Bu^n)\delta W^n, \tilde{\Bu}^{n+1}-\Bu^n-g(\Bu^n)\delta W^n)}, \label{eq:SAV4}\\
	\Bu^0=\Bu_0, \ \xi^0=1, \ \eta^0=1,
\end{subnumcases}
where $\delta W^n=W(t_{n+1})-W(t_n)$.

Next, we explain how to solve this system effectively.  

The first step is to
set $\tilde{\Bu}^{n+1}=\tilde{\Bu}_1^{n+1}+\xi^{n+1}\tilde{\Bu}_2^{n+1}+\eta^{n+1}\tilde{\Bu}_3^{n+1}$ in \eqref{eq:SAV1} and solve $\tilde{\Bu}_i^{n+1}, (i=1,2,3)$ from the following decoupled equations
\begin{subnumcases}
	{}\tilde{\Bu}^{n+1}_1-\Bu^n=\tau\Delta \tilde{\Bu}_1^{n+1}-\tau\nabla p^n, \ \ \ \tilde{\Bu}_1^{n+1}|_{\partial\mathcal{O}}=0, \\
	\tilde{\Bu}_2^{n+1}+\tau(\Bu^n\cdot\nabla)\Bu^n=\tau\Delta \tilde{\Bu}_2^{n+1}, \ \ \tilde{\Bu}_2^{n+1}|_{\partial\mathcal{O}}=0, \label{eq:tildeu2}\\
	\tilde{\Bu}_3^{n+1}=\tau\Delta \tilde{\Bu}_3^{n+1}+g(\Bu^n)\delta W^n, \ \ \ \tilde{\Bu}_3^{n+1}|_{\partial\mathcal{O}}=0. \label{eq:tildeu3}
\end{subnumcases}

Next, set 
\begin{align}
	{\Bu}^{n+1}&={\Bu}_1^{n+1}+\xi^{n+1}{\Bu}_2^{n+1}+\eta^{n+1}{\Bu}_3^{n+1}, \label{eq:semiu}\\
	p^{n+1}&=p_1^{n+1}+\xi^{n+1} p_2^{n+1}+\eta^{n+1} p_3^{n+1}, \label{eq:semip}
\end{align}
and solve $\Bu^{n+1}_i, p_i^{n+1}, (i=1,2,3)$ from  
\begin{subnumcases}
	{}\Bu_1^{n+1}-\tilde{\Bu}_1^{n+1}+\tau(\nabla p_1^{n+1}-\nabla p^n)=0,  \\
	\Bu_2^{n+1}-\tilde{\Bu}_2^{n+1}+\tau\nabla p_2^{n+1}=0,  \\
	\Bu_3^{n+1}-\tilde{\Bu}_3^{n+1}+\tau\nabla p_3^{n+1}=0,  \\
	\nabla\cdot \Bu_i^{n+1}=0, \ \ \ \Bu_i^{n+1}\cdot {\bf n}|_{\partial\mathcal{O}}=0.
\end{subnumcases}

Based upon the computed $\tilde{\Bu}_i^{n+1}, \Bu_i^{n+1}, p_i^{n+1}$, we can determine $\xi^{n+1}$ and $\eta^{n+1}$ by substituting expressions of $\tilde{\Bu}^{n+1}$ and $\Bu^{n+1}$ into \eqref{eq:SAV3} and \eqref{eq:SAV4} and solve the following linear system:
$$AX=b,$$
where 
\begin{align}
	&A=\begin{pmatrix}
		1-\tau([\Bu^n\cdot\nabla]\Bu^n, \tilde{\Bu}_2^{n+1}) & -\tau([\Bu^n\cdot\nabla]\Bu^n, \tilde{\Bu}_3^{n+1})\\
		(g(\Bu^n)\delta W^n, \tilde{\Bu}_2^{n+1})& 1+(g(\Bu^n)\delta W^n, \tilde{\Bu}_3^{n+1})
	\end{pmatrix}, \quad\ \ 
	X=\begin{pmatrix}
		\xi^{n+1} \\
		\eta^{n+1}
	\end{pmatrix}, \\
	&b=\begin{pmatrix}
		\xi^n+\tau([\Bu^n\cdot\nabla]\Bu^n, \tilde{\Bu}_1^{n+1})\\
		\eta^n-(g(\Bu^n)\delta W^n, \tilde{\Bu}_1^{n+1}-\Bu^n-g(\Bu^n)\delta W^n)
	\end{pmatrix}.
\end{align}

However, it is not easy to determine whether $A$ is singular or not. Towards this end, 
taking inner product of both sides of \eqref{eq:tildeu2} with $\tilde{\Bu}_2^{n+1}$ and $\tilde{\Bu}_3^{n+1}$ respectively results in
\begin{align}
	&-\tau([\Bu^n\cdot\nabla]\Bu^n, \tilde{\Bu}_2^{n+1}) =\tau\|\nabla\tilde{\Bu}_2^{n+1}\|^2+\|\tilde{\Bu}^{n+1}_2\|^2, \\
	&-\tau([\Bu^n\cdot\nabla]\Bu^n, \tilde{\Bu}_3^{n+1})=\tau(\nabla\tilde{\Bu}_2^{n+1}, \nabla\tilde{\Bu}_3^{n+1})+(\tilde{\Bu}_2^{n+1}, \tilde{\Bu}_3^{n+1}).
\end{align}
Similarly, take inner product of both sides of \eqref{eq:tildeu3} with $\tilde{\Bu}_3^{n+1}$ and  $\tilde{\Bu}_2^{n+1}$  leads to
\begin{align}
	&(g(\Bu^n)\delta W^n, \tilde{\Bu}_3^{n+1}) =\tau\|\nabla\tilde{\Bu}_3^{n+1}\|^2+\|\tilde{\Bu}^{n+1}_3\|^2, \\
	&(g(\Bu^n)\delta W^n, \tilde{\Bu}_2^{n+1})=\tau(\nabla\tilde{\Bu}_2^{n+1}, \nabla\tilde{\Bu}_3^{n+1})+(\tilde{\Bu}_2^{n+1}, \tilde{\Bu}_3^{n+1}).
\end{align}
Hence, $A$ has the equivalent form
\begin{equation}
	A=\begin{pmatrix}
		1+\tau\|\nabla\tilde{\Bu}_2^{n+1}\|^2+\|\tilde{\Bu}^{n+1}_2\|^2 & \tau(\nabla\tilde{\Bu}_2^{n+1}, \nabla\tilde{\Bu}_3^{n+1})+(\tilde{\Bu}_2^{n+1}, \tilde{\Bu}_3^{n+1})\\
		\tau(\nabla\tilde{\Bu}_2^{n+1}, \nabla\tilde{\Bu}_3^{n+1})+(\tilde{\Bu}_2^{n+1}, \tilde{\Bu}_3^{n+1}) & 1+\tau\|\nabla\tilde{\Bu}_3^{n+1}\|^2+\|\tilde{\Bu}^{n+1}_3\|^2
	\end{pmatrix}.
\end{equation}
A simple application of Cauchy-Schwarz inequality shows that $A$ is symmetric and positive-definite. Hence, the $2\times 2$ linear system is uniquely solvable.
Once $\xi^{n+1}$ and $\eta^{n+1}$ are obtained, we can use \eqref{eq:semiu} and \eqref{eq:semip} to update $\Bu^{n+1}$ and $p^{n+1}$. 

It is noteworthy that, at each time step, only $6$ linear systems and a $2 \times 2$ linear algebraic system need to be solved. Furthermore, the computed solution $\Bu^{n+1}$ enjoys the following unconditional stability property.

\subsection{Unconditional stability}
It is easy to show that our method is unconditionally stable up to $2m$ moments. More specifically, we have
\begin{lemma}\label{lem:sta}
	Let Assumption \ref{assump:1} be fulfilled. The auxiliary variable-based pressure-correction scheme \eqref{eq:SAV1}-\eqref{eq:SAV4} is unconditionally stable in the sense that
	\begin{align*}
		\mathbf{E}\|\Bu^{n+1}\|^{2m}+\mathbf E \bigg(\tau \sum\limits_{j=0}^n\|\nabla \tilde{\Bu}^{j+1}\|^2\bigg)^m+\mathbf{E}|\xi^{n+1}|^{2m}+\mathbf{E}|\eta^{n+1}|^{2m}+\tau^{2m} \mathbf{E}\|\nabla p^{n+1}\|^{2m}<\infty, \ \ \forall m\geq 1. 
	\end{align*}
\end{lemma}
\begin{proof}
	Taking inner product of \eqref{eq:SAV1} with $2\tilde{\Bu}^{n+1}$ and using the basic identity $(a-b, 2a)=\|a\|^2-\|b\|^2+\|a-b\|^2$ lead to
	\begin{align}
		&\|\tilde{\Bu}^{n+1}\|^2-\|\Bu^n\|^2+\|\tilde{\Bu}^{n+1}-\Bu^n\|^2+2\tau\xi^{n+1}([\Bu^n\cdot\nabla]\Bu^n, \tilde{\Bu}^{n+1})+2\tau\|\nabla\tilde{\Bu}^{n+1}\|^2+2\tau(\nabla p^n, \tilde{\Bu}^{n+1})\notag\\
		&=2\eta^{n+1}(g(\Bu^n)\delta W^n, \tilde{\Bu}^{n+1}). \label{eq:sta1}
	\end{align}
	Next, we rewrite \eqref{eq:SAV2} as 
	\begin{align}
		\Bu^{n+1}+\tau\nabla p^{n+1}=\tilde{\Bu}^{n+1}+\tau\nabla p^n.
	\end{align}
	Taking inner product of both sides with itself and using \eqref{eq:SAV21}  gives
	\begin{align}
		\|\Bu^{n+1}\|^2+\tau^2\|\nabla p^{n+1}\|^2=\|\tilde{\Bu}^{n+1}\|^2+2\tau (\tilde{\Bu}^{n+1}, \nabla p^n)+\tau^2\|\nabla p^n\|^2. \label{eq:sta2}
	\end{align}
	Similarly, multiplying both sides of \eqref{eq:SAV3} by $2\xi^{n+1}$ and multiplying both sides of \eqref{eq:SAV4} by $2\eta^{n+1}$ give
	\begin{align}
		&|\xi^{n+1}|^2-|\xi^n|^2+|\xi^{n+1}-\xi^n|^2=2\tau \xi^{n+1}([\Bu^n\cdot\nabla]\Bu^n, \tilde{\Bu}^{n+1}), \label{eq:sta3} \\
		&|\eta^{n+1}|^2-|\eta^n|^2+|\eta^{n+1}-\eta^n|^2=-2\eta^{n+1}(g(\Bu^n)\delta W^n, \tilde{\Bu}^{n+1}-\Bu^n-g(\Bu^n)\delta W^n). \label{eq:sta4}
	\end{align}
	Summing up \eqref{eq:sta1}, \eqref{eq:sta2}, \eqref{eq:sta3} and \eqref{eq:sta4}, we derive
	\begin{align}
		&\|\Bu^{n+1}\|^2-\|\Bu^n\|^2+\|\tilde{\Bu}^{n+1}-\Bu^n\|^2+2\tau\|\nabla \tilde{\Bu}^{n+1}\|^2+\tau^2\|\nabla p^{n+1}\|^2 \notag\\
		&+|\xi^{n+1}|^2-|\xi^n|^2+|\xi^{n+1}-\xi^n|^2+|\eta^{n+1}|^2-|\eta^n|^2+|\eta^{n+1}-\eta^n|^2\notag\\
		&=\tau^2\|\nabla p^n\|^2+2(\eta^{n+1}-\eta^n)(g(\Bu^n)\delta W^n, \Bu^n)+2\eta^n(g(\Bu^n)\delta W^n, \Bu^n)\notag\\
		&\quad+2(\eta^{n+1}-\eta^n)\|g(\Bu^n)\delta W^n\|^2+2\eta^n\|g(\Bu^n)\delta W^n\|^2\notag\\
		&\leq \tau^2\|\nabla p^n\|^2+\frac{1}{2}|\eta^{n+1}-\eta^n|^2
		+4|(g(\Bu^n)\delta W^n, \Bu^n)|^2+2(g(\Bu^n)\delta W^n, \eta^n\Bu^n)\notag\\
		&\quad+4\|g(\Bu^n)\delta W^n\|^4+2\eta^n\|g(\Bu^n)\delta W^n\|^2,
	\end{align}
	where the Young's inequality $2ab\leq 4a^2+\frac{1}{4}b^2$ has been used twice. 
	
	Hence, summing over $n$ implies
	\begin{align}
		&\|\Bu^{n+1}\|^2+\tau \sum\limits_{j=0}^n\|\nabla \tilde{\Bu}^{j+1}\|^2+\tau^2 \|\nabla p^{n+1}\|^2+|\xi^{n+1}|^2+|\eta^{n+1}|^2\notag\\
		&\leq \|\Bu^0\|^2+|\xi^0|^2+|\eta^0|^2+\tau^2\|\nabla p^0\|^2+4\sum\limits_{j=0}^n|(g(\Bu^j)\delta W^j, \Bu^j)|^2+2\sum\limits_{j=0}^n(g(\Bu^j)\delta W^j, \eta^j\Bu^j)\notag\\
		&\quad+4\sum\limits_{j=0}^n\|g(\Bu^j)\delta W^j\|^4+2\sum\limits_{j=0}^n\eta^j\|g(\Bu^j)\delta W^j\|^2. 
	\end{align}
	Taking $m$-th power and then expectation on both sides, and applying the elementary inequality
	\begin{align}\label{eq:elem1}
		|a_1+a_2+\cdots+a_n|^p\leq n^{p-1}\sum\limits_{j=1}^n |a_j|^p,\ \  p\geq 1,
	\end{align}
	we can obtain
	\begin{align}\label{eq:stan1}
		&\mathbf{E}\|\Bu^{n+1}\|^{2m}+\mathbf E \bigg(\tau \sum\limits_{j=0}^n\|\nabla \tilde{\Bu}^{j+1}\|^2\bigg)^m+\mathbf{E}|\xi^{n+1}|^{2m}+\mathbf{E}|\eta^{n+1}|^{2m}+\tau^{2m} \mathbf{E}\|\nabla p^{n+1}\|^{2m}\notag\\
		&\leq C+\tau^{2m}\mathbf{E}\|\nabla p^0\|^{2m}+C\mathbf{E}\bigg[\sum\limits_{j=0}^n|(g(\Bu^j)\delta W^j, \Bu^j)|^2\bigg]^m+C\mathbf{E}\bigg|\sum\limits_{j=0}^n(g(\Bu^j)\delta W^j, \Bu^j\eta^j)\bigg|^m\notag\\
		&\quad+C\mathbf{E}\bigg[\sum\limits_{j=0}^n\|g(\Bu^j)\delta W^j\|^4\bigg]^m+C\mathbf{E}\bigg[\sum\limits_{j=0}^n \eta^j\|g(\Bu^j)\delta W^j\|^2\bigg]^m\notag\\
		&:=C(\|\Bu_0\|,m)+\tau^{2m}\mathbf{E}\|\nabla p^0\|^{2m}+S_1+S_2+S_3+S_4.
	\end{align}
	Next, we estimate  terms $S_i, (i=1,2,3,4)$ individually. By Assumption \ref{assump:1} 
	we derive
	\begin{align}
		S_1&\leq C\mathbf{E}\bigg[\sum\limits_{j=0}^n\|\delta W^j\|^2\|\Bu^j\|^2\bigg]^m\leq Cn^{m-1}\sum\limits_{j=0}^n \mathbf{E}[\|\delta W^j\|^{2m}\|\Bu^j\|^{2m}] \notag\\
		&\leq Cn^{m-1}\sum\limits_{j=0}^n \mathbf{E}\|\delta W^j\|^{2m}\mathbf{E}\|\Bu^j\|^{2m}\notag\\
		&\leq C(T,m)\tau\sum\limits_{j=0}^n \mathbf{E}\|\Bu^j\|^{2m}.
	\end{align}
	Based upon the discrete BDG ineuality and Assumption \ref{assump:1}, we can have
	\begin{align}
		S_2&\leq C\bigg[\sum\limits_{j=0}^n \bigg(\mathbf{E}|(g(\Bu^j)\delta W^j, \Bu^j\eta^j)|^m\bigg)^{\frac{2}{m}}\bigg]^{\frac{m}{2}}\notag\\
		&\leq C\bigg[\sum\limits_{j=0}^n \bigg(\mathbf{E}(\|\delta W^j\|^m \|\Bu^j\|^m|\eta^j|^m)\bigg)^{\frac{2}{m}}\bigg]^{\frac{m}{2}}\notag\\
		&\leq C\bigg[\sum\limits_{j=0}^n \bigg(\mathbf{E}( \|\Bu^j\|^m|\eta^j|^m\tau^{\frac{m}{2}})\bigg)^{\frac{2}{m}}\bigg]^{\frac{m}{2}}\notag\\
		&\leq C\bigg[\tau\sum\limits_{j=0}^n \bigg(\mathbf{E}( \|\Bu^j\|^{2m}+\mathbf{E}|\eta^j|^{2m})\bigg)^{\frac{2}{m}}\bigg]^{\frac{m}{2}}\notag\\
		&\leq  C\bigg[\tau\sum\limits_{j=0}^n \big(\mathbf{E}( \|\Bu^j\|^{2m}\big)^{\frac{2}{m}}+\big(\mathbf{E}|\eta^j|^{2m})\big)^{\frac{2}{m}}\bigg]^{\frac{m}{2}},
	\end{align}
	where the following elementary inequality has been used:
	$$(a+b)^{\frac{2}{m}}\leq 2(a^{\frac{2}{m}}+b^{\frac{2}{m}}), \ a, b\geq 0, \ m\geq 1.$$
	Employing \eqref{eq:elem1} again, we directly have
	\begin{align}
		S_2&\leq C(T,m, Tr(Q))\tau (\mathbf{E} \|\Bu^j\|^{2m}+\mathbf{E}|\eta^j|^{2m}).
	\end{align}
	Following a similar manner, 
	\begin{align}
		S_3&\leq C\mathbf{E}\bigg[\sum\limits_{j=0}^n \|\delta W^j\|^4\bigg]^m\leq Cn^{m-1}\sum\limits_{j=0}^n\mathbf{E}\|\delta W^j\|^{4m}\leq C(T, Tr(Q),m)\tau
	\end{align}
	and 
	\begin{align}
		S_4&\leq C\mathbf{E}\bigg[\sum\limits_{j=0}^n |\eta^j|\|\delta W^j\|^2\bigg]^m\leq Cn^{m-1}\sum\limits_{j=0}^n \mathbf{E}(|\eta^j|^m\|\delta W^j\|^{2m}) \notag\\
		&\leq Cn^{m-1}\sum\limits_{j=0}^n \mathbf{E}|\eta^j|^m \mathbf{E}\|\delta W^j\|^{2m} \leq C(T,m, Tr(Q))\tau\sum\limits_{j=0}^n \mathbf{E}|\eta^j|^{2m}.
	\end{align}
	Substituting the estimates of $S_1$-$S_4$ into \eqref{eq:stan1}, we have
	\begin{align}
		&\mathbf{E}\|\Bu^{n+1}\|^{2m}+\mathbf E \bigg(\tau \sum\limits_{j=0}^n\|\nabla \tilde{\Bu}^{j+1}\|^2\bigg)^m+\mathbf{E}|\xi^{n+1}|^{2m}+\mathbf{E}|\eta^{n+1}|^{2m}+\tau^{2m} \mathbf{E}\|\nabla p^{n+1}\|^{2m}\\
		&\leq C(\|\Bu_0\|,m)+\tau^{2m}\mathbf{E}\|\nabla p^0\|^{2m}+C(T,m, Tr(Q))\tau\sum\limits_{j=0}^n \mathbf{E}|\eta^j|^{2m}+C(T,m, Tr(Q))\tau\sum\limits_{j=0}^n \mathbf{E}|\Bu^j|^{2m}. \notag
	\end{align}
	The desired result is achieved by the discrete Gronwall's inequality.
\end{proof}
\begin{rem}\label{rem1}
	In contrast to \eqref{eq:SAV4}, an alternative discretization of the governing equation for $\eta$ is given by
	$$\tilde{\eta}^{n+1}-\tilde{\eta}^n=-(g(\Bu^n)\delta W^n, \tilde{\Bu}^{n+1}-\Bu^n).$$
	By following the same approach as in our proof, one can derive the stability result stated in our lemma. However, \eqref{eq:SAV4} incorporates a mean-reverting mechanism, which is analogous to \eqref{eq:illeta} below. Furthermore, this mechanism facilitates the development of the error analysis, as will be demonstrated in the next section.
\end{rem}

\subsection{Why two auxiliary variables}
As inferred from \cite{LiSL21}, the auxiliary variable $\xi$ is employed to handle the nonlinear term in \eqref{model:NS}, analogous to its deterministic counterpart. By utilizing this approach, we are able to handle the nonlinear term explicitly while preserving unconditional stability. In this subsection, we will demonstrate the necessity of the auxiliary variable $\eta$ through the following Langevin stochastic differential equation:
\begin{align}
	dX(t)=-\nabla V(x)dt+\sqrt{2}dB(t), \ \  X(0)=x_0, \ \ t\in [0,T],
\end{align}
where $V(x)=\frac{1}{4}(x^2-1)^2$ is a double-well potential function. It is well-known that the invariant measure of this process 
$$
\pi(x)\propto \exp(-V(x)). 
$$

We shall make a comparison of the following one-auxiliary variable (OAV) method and two-auxiliary variable (TAV) method for the SDE ($\tau$ denotes time stepsize and $\delta B^n\sim\sqrt{\tau}\mathscr{N}(0,1)$).

(1) OAV method:
\begin{subnumcases}
	{}X^{n+1}-X^n=-\tau\xi^{n+1}\nabla V(X^n)+\sqrt{2}\delta B^n, \\
	\xi^{n+1}-\xi^n=\tau(\nabla V(X^n), X^{n+1}-X^n+\tau \nabla V(X^n)), \\
	X^0=x_0,\  \xi^0=1; 
\end{subnumcases}

(2) TAV method:
\begin{subnumcases}
	{}X^{n+1}-X^n=-\tau\xi^{n+1}\nabla V(X^n)+\sqrt{2}\eta^{n+1}\delta B^n, \label{eq:illx}\\
	\xi^{n+1}-\xi^n=\tau(\nabla V(X^n), X^{n+1}-X^n+\tau \nabla V(X^n)), \label{eq:illxi}\\
	\eta^{n+1}-\eta^n=-(\sqrt{2}\delta B^n, X^{n+1}-X^n-\sqrt{2}\delta B^n), \label{eq:illeta}\\
	X^0=x_0,\  \xi^0=1,\ \eta^0=1. 
\end{subnumcases}

We have added a mean-reverting mechanism (cf. \cite{ColemanHW25}) in both methods such that larger step sizes are permitted. More specifically, in the latter method, we substitute \eqref{eq:illx} into \eqref{eq:illxi} and \eqref{eq:illeta} and get
\begin{align}
	&\xi^{n+1}-\sqrt{2}\tau \delta B^n \nabla V(X^n) \eta^{n+1}=\xi^n-\tau^2\nabla V(X^n)^2 (\xi^{n+1}-1);\notag\\
	&\eta^{n+1}-\sqrt{2}\tau \delta B^n \nabla V(X^n) \xi^{n+1}=\eta^n-2|\delta B^n|^2(\eta^{n+1}-1).
\end{align}
Apparently, $-\tau^2\nabla V(X^n)^2 (\xi^{n+1}-1)$ and $-2|\delta B^n|^2(\eta^{n+1}-1)$ are  mean-reverting terms for $\xi^{n+1}$ and $\eta^{n+1}$ respectively. 
Similar pattern can be easily observed for OAV method. 

In our numerical experiments, we take $T=20, x_0\sim 10 \mathscr{N}(0,1)$. Moreover, we repeat $50000$ times for  Monte Carlo approximation and use the Kullback-Leibler (KL) divergence to measure the distance of histograms produced for both methods and $\pi(x)$.

\renewcommand{\arraystretch}{1.5}
\begin{table}
	\begin{tabular}{ c|c|c } 
		\hline
		$\tau$ & OAV method & TAV method \\ 
		\hline
		$\frac{T}{200} $& 0.3521 & 0.0290 \\ 
		$\frac{T}{400}$  & 0.2914 & 0.0591 \\ 
		$\frac{T}{800} $&0.0151 & 0.0020\\
		$\frac{T}{1600}$ & 0.0009& 0.0011\\ 
		$\frac{T}{3200}$ & 0.0005 & 0.0006\\
		\hline
	\end{tabular}
	\caption{KL divergence between numerical results of both methods and $\pi(x)$. }\label{table31}
\end{table}

As evident from Table \ref{table31}, the TAV method demonstrates a clear advantage in terms of KL divergence for large step sizes. This is further illustrated by the histogram plot in Fig \ref{fig1}. However, for small time step sizes, the two methods exhibit only minor differences.
\begin{figure}[ht]
	\centering
	\vskip -1in
	\resizebox{100mm}{130mm}{\includegraphics{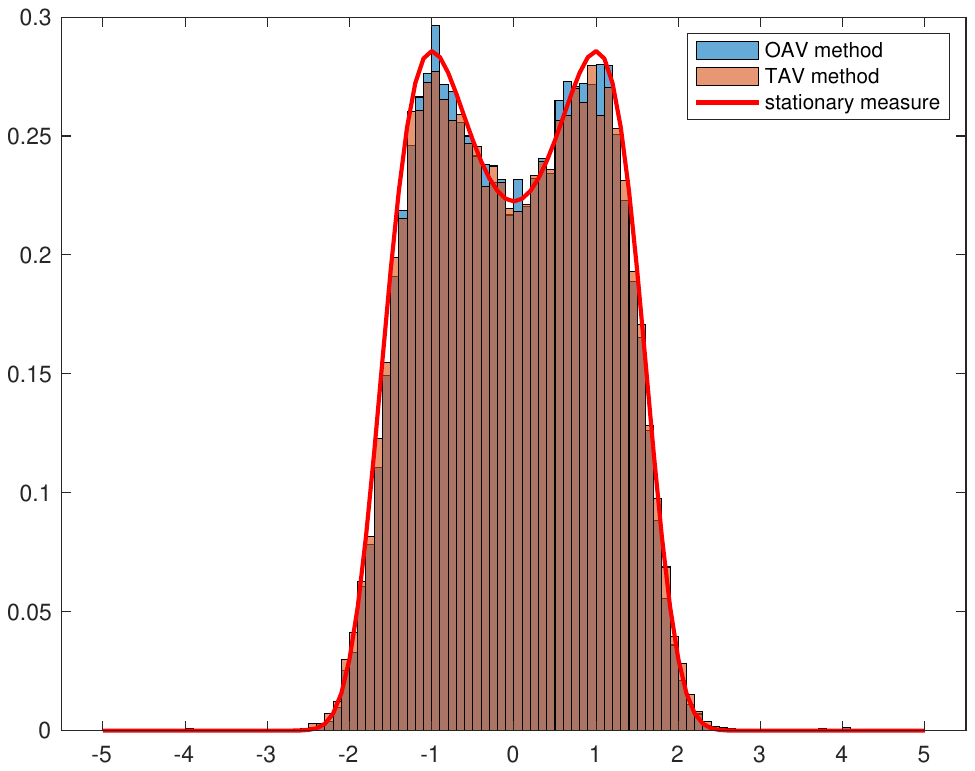}} \vskip -1.4in
	\caption{Histogram produced by OAV method (KL divergence=$0.0190$) and TAV method (KL divergence=$0.0027$) for $\tau=T/700$.} \label{fig1}
\end{figure}
\section{Error analysis for linearized equation}

In this section, we will conduct a convergence analysis for our auxiliary variable-based pressure-correction scheme, leveraging the unconditional stability result from Lemma \ref{lem:sta}. Due to the complexity of the convection term, which interacts with the noise, many error analyses for numerical schemes applied to the stochastic Navier-Stokes equations have been performed locally (cf. \cite{BessaihM22, BreitD21, CarelliP12, HausenblasR19}). To mitigate this challenge and take advantage of established regularity results, we provide an error analysis for our auxiliary variable scheme applied to the linearized stochastic Navier-Stokes equations. Specifically, we focus on the error analysis for the following equation.
\begin{equation}\label{model:NS1}
	\begin{cases}
		\displaystyle{d{\Bu}=\Delta {\Bu}dt-[\Bv\cdot\nabla]\Bu dt-\nabla p dt+g({\Bu})d{ W}(t), \;\; a.s. \  \Bx\in \mathcal{O}=[0,1]^2,}\\
		\text{div}\  {\Bu}=0, \;\; \ a.s. ,\\
		{\Bu}(0)=\Bu_0,\;\; \ a.s. ,
	\end{cases}
\end{equation}
where we assume that $\Bv\in H^2$ w.p.1, and is independent of $\Bu$. Our motivation of this treatment resides in the governing equations of Oseen's flow at low Reynolds numbers where $\Bv$ is considered as a constant vector (cf. \cite{Batchelor92, Oseen27}), as well as similar approach in the numerical treatment of deterministic Navier-Stokes equations (cf. \cite{LLS00}).

Based upon \eqref{model:NS1}, the equation \eqref{eq:SAV1} in our scheme is accordingly replaced by
\begin{align}\label{eq:SAV1b}
	\tilde{\Bu}^{n+1}-\Bu^n=\tau\Delta \tilde{\Bu}^{n+1}-\tau\xi^{n+1} [\Bv^n\cdot\nabla]\Bu^{n}-\tau\nabla p^{n}+\eta^{n+1}g(\Bu^n)\delta W^n
\end{align}
and \eqref{eq:SAV2}-\eqref{eq:SAV4} remain unchanged. 

Let $(\tilde{\Bu}^{n+1},\Bu^{n+1}, p^{n+1},\xi^{n+1}, \eta^{n+1})$ be solutions of \eqref{eq:SAV1}-\eqref{eq:SAV4}.  Denote
\begin{align*}
	&\tilde{e}_\Bu^{n+1}=\tilde{\Bu}^{n+1}-\Bu(t_{n+1}),  \ e_\Bu^{n+1}=\Bu^{n+1}-\Bu(t_{n+1}),\  e_p^{n+1}=p^{n+1}-p(t_{n+1}), \\
	&\ e_\xi^{n+1}=\xi^{n+1}-\xi(t_{n+1}), \ e_{\eta}^{n+1}=\eta^{n+1}-\eta(t_{n+1}), \ e_\Bv^{n+1}=\Bv^{n+1}-\Bv(t_{n+1}).
\end{align*}
Clearly, $\text{div}\  e_\Bu^{n+1}=0$ and $e_\Bu^0=\tilde{e}_\Bu^0=0$.

\begin{theorem}
	Suppose that Assumptions \ref{assump:1}- \ref{assump:3} are fulfilled. Then, our scheme for the linearized stochastic N-S equation \eqref{model:NS1} has the following error bound:
	\begin{align}\label{eq:final}
		\mathbf{E}\| \Bu(t_{n+1})-\Bu^{n+1}\|^2\leq C\tau^{2\theta}, \ \ \forall\  0<\theta<\frac{1}{2},\ \ 1\leq n\leq \frac{T}{\tau},
	\end{align} 
	where $C$ is independent of $\tau$, but depends on $u_0, T, C_g$.
\end{theorem}
\begin{proof}
	
	We shall use the discrete Gronwall's inequality twice in the proof, and thus divide the proof into two parts. 
	
	{\bf Part 1: Preparations before taking expectation}
	
	Integrating \eqref{model:NS1} over $[t_n, t_{n+1}]$, we have
	\begin{align}\label{eq:u1}
		&\Bu(t_{n+1})-\Bu(t_n)+\int_{t_n}^{t_{n+1}}\xi(s)[\Bv(s)\cdot\nabla] \Bu(s)ds -\int_{t_n}^{t_{n+1}} \Delta\Bu(s)ds+\int_{t_n}^{t_{n+1}}\nabla p(s)ds\notag\\
		&=\int_{t_n}^{t_{n+1}} \eta(s)g(\Bu(s))dW(s).
	\end{align}
	Subtracting \eqref{eq:u1} from \eqref{eq:SAV1b}, we can have
	\begin{align}
		& \tilde{e}_\Bu^{n+1}-e_\Bu^n-\tau\Delta \tilde{\Be}_\Bu^{n+1}+\int_{t_n}^{t_{n+1}}\xi^{n+1}[{\Bv}^{n}\cdot\nabla]{\Bu}^{n}-\xi(s)[\Bv(s)\cdot\nabla]\Bu(s)ds+\tau\nabla p^n-\int_{t_n}^{t_{n+1}}\nabla p(s)ds
		\notag\\
		&\qquad=\int_{t_n}^{t_{n+1}} \Delta {\Bu}(t_{n+1})-\Delta \Bu(s)ds+\int_{t_n}^{t_{n+1}} [\eta^{n+1}g(\Bu^n)-\eta(s)g(\Bu(s))]dW(s).
	\end{align}
	Taking inner product with $2\tilde{e}_\Bu^{n+1}$ and integration by parts,  we derive
	\begin{align}\label{eq:e1}
		&\|\tilde{e}_\Bu^{n+1}\|^2-\|e_\Bu^n\|^2+\|\tilde{e}_\Bu^{n+1}-e_\Bu^n\|^2+2\tau\|\nabla \tilde{e}_\Bu^{n+1}\|^2+\bigg(\tau\nabla p^n-\int_{t_n}^{t_{n+1}}\nabla p(s)ds, \ 2\tilde{e}_\Bu^{n+1}\bigg) \notag\\
		&=-2\int_{t_n}^{t_{n+1}}(\xi^{n+1}[{\Bv}^{n}\cdot\nabla]{\Bu}^{n}-\xi(s)[\Bv(s)\cdot\nabla]\Bu(s), \tilde{e}_\Bu^{n+1})ds
		+2\int_{t_n}^{t_{n+1}} (\nabla \Bu(s)-\nabla {\Bu}(t_{n+1}), \nabla\tilde{e}_\Bu^{n+1})ds\notag\\
		&\quad\quad +2\bigg(\int_{t_n}^{t_{n+1}} \eta^{n+1}g(\Bu^n)-\eta(s)g(\Bu(s))dW(s), \tilde{e}_\Bu^{n+1}\bigg).
	\end{align}
	Moreover, we obtain from \eqref{eq:SAV2} that
	\begin{align}\label{eq:eetilde}
		e_\Bu^{n+1}-\tilde{e}_\Bu^{n+1}+\tau (\nabla p^{n+1}-\nabla p^n)=0,
	\end{align}
	which taking inner product with $e_\Bu^{n+1}+\tilde{e}_\Bu^{n+1}$ implies
	\begin{align}\label{eq:e2}
		\|e_\Bu^{n+1}\|^2-\|\tilde{e}_\Bu^{n+1}\|^2+\tau (\nabla p^{n+1}-\nabla p^n,  \tilde{e}_\Bu^{n+1})=0. 
	\end{align}
	Next, summing up \eqref{eq:e1} and \eqref{eq:e2} gives
	\begin{align}\label{eq:e3}
		&\|{e}_\Bu^{n+1}\|^2-\|e_\Bu^n\|^2+\|\tilde{e}_\Bu^{n+1}-e_\Bu^n\|^2+2\tau\|\nabla \tilde{e}_\Bu^{n+1}\|^2+\bigg(\tau\nabla p^{n+1}+\tau\nabla p^n-2\int_{t_n}^{t_{n+1}}\nabla p(s)ds, \tilde{e}_\Bu^{n+1}\bigg) \notag\\
		&\quad=-2\int_{t_n}^{t_{n+1}}(\xi^{n+1}[{\Bv}^{n}\cdot\nabla]{\Bu}^{n}-\xi(s)[\Bv(s)\cdot\nabla]\Bu(s), \tilde{e}_\Bu^{n+1})ds
		+2\int_{t_n}^{t_{n+1}} (\nabla \Bu(s)-\nabla {\Bu}(t_{n+1}), \nabla\tilde{\Be}_u^{n+1})ds\notag\\
		&\quad\quad +2\bigg(\int_{t_n}^{t_{n+1}} \eta^{n+1}g(\Bu^n)-\eta(s)g(\Bu(s))dW(s), \tilde{e}_\Bu^{n+1}\bigg).
	\end{align}
	Now, we are in a position to bound terms in \eqref{eq:e3} in turn. 
	
	{\it (i) Estimate of the pressure term}
	
	Since $\nabla\cdot e_\Bu^{n+1}=0$, we have by \eqref{eq:eetilde} and integration by parts
	\begin{align}
		&\bigg(\tau\nabla p^{n+1}+\tau\nabla p^n-2\int_{t_n}^{t_{n+1}}\nabla p(s)ds, \tilde{e}_\Bu^{n+1}\bigg) \notag\\
		&=\bigg(\tau \nabla p^{n+1}+\tau \nabla p^n-2\int_{t_n}^{t_{n+1}} \nabla p(s)ds, \ e_\Bu^{n+1}+\tau (\nabla p^{n+1}-\nabla p^n)\bigg)\notag\\
		&=\bigg(\tau \nabla p^{n+1}+\tau \nabla p^n-2\int_{t_n}^{t_{n+1}} \nabla p(s)ds, \tau (\nabla p^{n+1}-\nabla p^n)\bigg)\notag
	\end{align}
	\begin{align}
		&=\bigg(\tau \nabla e_p^{n+1}+\tau\nabla e_p^n+\int_{t_n}^{t_{n+1}} (\nabla p(t_{n+1})-\nabla p(s))ds+\int_{t_n}^{t_{n+1}} (\nabla p(t_{n})-\nabla p(s))ds, \notag\\
		&\qquad \tau (\nabla e_p^{n+1}-\nabla e_p^n)+\tau (\nabla p(t_{n+1})-\nabla p(t_n))\bigg)\notag\\
		&=\tau^2\|\nabla e_p^{n+1}\|^2-\tau^2\|\nabla e_p^n\|^2+I_1+I_2+I_3.
	\end{align}
	We have noted that $I_1, I_2, I_3$ will be moved to the right-hand side of \eqref{eq:e3}. By the Young's inequality, we easily deduce that
	\begin{align}
		|I_1|&=|(\tau \nabla e_p^{n+1}+\tau\nabla e_p^n, \tau (\nabla p(t_{n+1})-\nabla p(t_n)))|\label{eq:eI1}\\
		&\leq \tau^3(\|\nabla e_p^{n+1}\|^2+\|\nabla e_p^n\|^2)+C\tau\|\nabla p(t_{n+1})-\nabla p(t_n)\|^2; \notag\\
		|I_2|&=\bigg|\bigg(\int_{t_n}^{t_{n+1}} (\nabla p(t_{n+1})-\nabla p(s))ds+\int_{t_n}^{t_{n+1}} (\nabla p(t_{n})-\nabla p(s))ds, 
		\tau (\nabla e_p^{n+1}-\nabla e_p^n)\bigg)\bigg|\label{eq:eI2}\\
		&\leq \tau^3(\|\nabla e_p^{n+1}\|^2+\|\nabla e_p^n\|^2)+C\tau\|\nabla p(t_{n+1})-\nabla p(t_n)\|^2; \notag\\
		|I_3|&=\bigg|\bigg(\int_{t_n}^{t_{n+1}} (\nabla p(t_{n+1})-\nabla p(s))ds+\int_{t_n}^{t_{n+1}} (\nabla p(t_{n})-\nabla p(s))ds, \tau (\nabla p(t_{n+1})-\nabla p(t_n))\bigg)\bigg|\notag\\
		&\leq C\tau^2\|\nabla p(t_{n+1})-\nabla p(t_n)\|^2. \label{eq:eI3}
	\end{align}
	{\it (ii) Estimate of the convection term}
	
	Using the fact that $\xi(s)\equiv 1$, we split the convection term as follows.
	\begin{align}\label{eq:econv1}
		&-\int_{t_n}^{t_{n+1}}(\xi^{n+1}[{\Bv}^{n}\cdot\nabla]{\Bu}^{n}-[\Bv(s)\cdot\nabla]\Bu(s), 2\tilde{e}_\Bu^{n+1})ds \notag\\
		&=-\int_{t_n}^{t_{n+1}}e_\xi^{n+1}([{\Bv}^{n}\cdot\nabla]{\Bu}^{n}, 2\tilde{e}_\Bu^{n+1})ds-\int_{t_n}^{t_{n+1}}([{\Bv}^{n}\cdot\nabla]{\Bu}^{n}-[\Bv(s)\cdot\nabla]\Bu(s), 2\tilde{e}_\Bu^{n+1})ds,
	\end{align}
	where the first term is difficult to bound. Fortunately, the discretization of auxiliary equation on $\xi$ \eqref{eq:SAV3} compensates for it. 
	
	We subtract both sides of \eqref{eq:SAV3} by the trivial identity $1-1=0$ and then multiply by $2e_\xi^{n+1}$ to get
	\begin{align}\label{eq:exi}
		|e_\xi^{n+1}|^2&-|e_\xi^n|^2+|e_\xi^{n+1}-e_\xi^n|^2=2\tau e_\xi^{n+1}  b(\Bv^n, \Bu^n, \tilde{\Bu}^{n+1}) \notag\\
		&=2\tau e_\xi^{n+1} b(\Bv^n, \Bu^n, \tilde{e}_\Bu^{n+1})+2\tau e_\xi^{n+1}b(\Bv^n, \Bu^n, \Bu(t_{n+1})-\Bu(t_n))-2\tau e_\xi^{n+1}b(\Bv^n,\Bu^n, e_\Bu^n)\notag\\
		&\leq 2\tau e_\xi^{n+1}b(\Bv^n, \Bu^n, \tilde{e}_\Bu^{n+1})+2\tau |e_\xi^{n+1}|\|\Bv_n\|_2 \|\Bu^n\| \|\Bu(t_{n+1})-\Bu(t_n)\|_1\notag\\
		&\quad-2\tau e_\xi^{n+1}b(\Bv^n, \Bu(t_n), e_\Bu^n),
	\end{align}
	where we have used the bound of the trilinear term \eqref{eq:bddb} and  the identity $b(\Bv^n, \Bu^n, \Bu^n)=0$.
	
	Thus, a combination of \eqref{eq:econv1} and \eqref{eq:exi} gives
	\begin{align}
		&|e_\xi^{n+1}|^2-|e_\xi^n|^2+|e_\xi^{n+1}-e_\xi^n|^2-\int_{t_n}^{t_{n+1}}(\xi^{n+1}[{\Bv}^{n}\cdot\nabla]{\Bu}^{n}-[\Bv(s)\cdot\nabla]\Bu(s), 2\tilde{e}_\Bu^{n+1})ds \notag\\
		&\leq -\int_{t_n}^{t_{n+1}}([{\Bv}^{n}\cdot\nabla]{\Bu}^{n}-[\Bv(s)\cdot\nabla]\Bu(s), 2\tilde{e}_\Bu^{n+1})ds+2\tau  |e_\xi^{n+1}|\|\Bv_n\|_2 \|\Bu^n\| \|\Bu(t_{n+1})-\Bu(t_n)\|_1\notag\\
		&\quad-2\tau e_\xi^{n+1}b(\Bv^n, \Bu(t_n), e_\Bu^n)\notag\\
		&:=J_1+J_2+J_3.
	\end{align}
	Next, let us bound these three terms.  Recall that $v$ is explicitly known and smooth. Hence, $\|\Bv(t)\|_2\leq C$, $e_\Bv^n=0$ and $\|\Bv(t_n)-\Bv(s)\|_2\leq C\tau, \ s\in [t_n, t_{n+1}]$.  Applying \eqref{eq:bddb} again,  we can have
	\begin{align}
		J_1&=-2\tau  b(e_\Bv^n, \Bu^n, \tilde{e}_\Bu^{n+1})-2\int_{t_n}^{t_{n+1}}b(\Bv(t_n)-\Bv(s), \Bu^n,\tilde{e}_\Bu^{n+1})ds\notag\\
		&\quad-2 \int_{t_n}^{t_{n+1}}b(\Bv(s), \Bu^n-\Bu(t_n),\tilde{e}_\Bu^{n+1})ds-2\int_{t_n}^{t_{n+1}} b(\Bv(s), \Bu(t_n)-\Bu(s),\tilde{e}_\Bu^{n+1})ds\notag\\
		&\leq 2c\int_{t_n}^{t_{n+1}} \|\Bv(t_n)-\Bv(s)\|_2\|\Bu^n\|\|\nabla\tilde{e}_\Bu^{n+1}\|\notag\\
		&\quad+\|\Bv(s)\|_2 \|\Bu^n-\Bu(t_n)\| \|\nabla\tilde{e}_\Bu^{n+1}\|+\|\Bv(s)\|_2 \|\Bu(t_n)-\Bu(s)\| \|\nabla\tilde{e}_\Bu^{n+1}\|ds\notag\\
		&\leq \frac{\tau}{2}\|\nabla \tilde{e}_\Bu^{n+1}\|^2+C\tau^2\|\Bu^n\|^2+C\tau \|e_\Bu^n\|^2+C\int_{t_n}^{t_{n+1}}\|\Bu(t_n)-\Bu(s)\|^2ds.
	\end{align}
	Similarly, we can obtain estimates of $J_2$ and $J_3$. 
	\begin{align}
		J_2&\leq C\tau  |e_\xi^{n+1}|\|\Bu^n\| \|\Bu(t_{n+1})-\Bu(t_n)\|_1\notag\\
		&\leq C\tau \|\Bu(t_{n+1})-\Bu(t_n)\|_1^2\|\Bu^n\|^2+C\tau |e_\xi^{n+1}|^2;\notag\\
		J_3&\leq C\tau |e_\xi^{n+1}| \|\Bu(t_n)\|_1\|e_\Bu^n\|\leq C\tau\|\Bu(t_n)\|_1^2\|e_\Bu^n\|^2+C\tau|e_\xi^{n+1}|^2.
	\end{align}
	
	Therefore,
	\begin{align}\label{eq:con}
		&{-}\int_{t_n}^{t_{n+1}}(\xi^{n+1}[{\Bv}^{n}\cdot\nabla]{\Bu}^{n}-[\Bv(s)\cdot\nabla]\Bu(s), 2\tilde{e}_\Bu^{n+1})ds+|e_\xi^{n+1}|^2-|e_\xi^n|^2+|e_\xi^{n+1}-e_\xi^n|^2 \notag\\
		&\leq  \frac{\tau}{2}\|\nabla \tilde{e}_\Bu^{n+1}\|^2+C\tau^2\|\Bu^n\|^2+C\tau \|e_\Bu^n\|^2+C\int_{t_n}^{t_{n+1}}\|\Bu(t_n)-\Bu(s)\|^2ds\notag\\
		&\quad+C\tau \|\Bu(t_{n+1})-\Bu(t_n)\|_1^2\|\Bu^n\|^2+C\tau |e_\xi^{n+1}|^2+C\tau\|\Bu(t_n)\|_1^2\|e_\Bu^n\|^2.
	\end{align}
	It is important to note that the last term is the primary source of difficulty in our estimate and necessitates special treatment.
	
	
	{\it (iii) Estimate of diffusion term}
	
	It is clear that
	\begin{align}\label{eq:diff}
		&2\int_{t_n}^{t_{n+1}} \|\nabla \Bu(s)-\nabla\Bu(t_{n+1})\| \|\nabla\tilde{e}_\Bu^{n+1}\|ds
		\leq 2\int_{t_n}^{t_{n+1}}\|\nabla \Bu(s)-\nabla\Bu(t_{n+1})\|^2ds+\frac{\tau}{2}\|\nabla \tilde{e}_\Bu^{n+1}\|^2.
	\end{align}
	
	{\it (iv) Estimate of stochastic term}
	
	Recall the fact that $\eta(s)\equiv 1$, 
	\begin{align}\label{eq:e5}
		&2\bigg(\int_{t_n}^{t_{n+1}} { [}\eta^{n+1}g(\Bu^n)-\eta(s)g(\Bu(s))dW(s), \tilde{e}_\Bu^{n+1}\bigg)\\
		&=2\bigg(\int_{t_n}^{t_{n+1}} e_\eta^{n+1}g(\Bu^n)dW(s), \tilde{e}_\Bu^{n+1}\bigg)+2\bigg(\int_{t_n}^{t_{n+1}} (g(\Bu^n)-g(\Bu(s)))dW(s), \tilde{e}_\Bu^{n+1}\bigg)\notag\\
		&=2\bigg(e_\eta^{n+1}g(\Bu^n)\delta W^n, \tilde{e}_\Bu^{n+1}-e_\Bu^n-\int_{t_n}^{t_{n+1}}[g(\Bu^n)-g(\Bu(s))]dW(s)\bigg)\notag\\
		&\quad+2\bigg(e_\eta^{n+1}g(\Bu^n)\delta W^n, {e}_\Bu^{n}+\int_{t_n}^{t_{n+1}}[g(\Bu^n)-g(\Bu(s))]dW(s)\bigg)\notag\\
		&\quad+2\bigg(\int_{t_n}^{t_{n+1}} (g(\Bu^n)-g(\Bu(s))){ dW(s)}, \tilde{e}_\Bu^{n+1}-e_\Bu^n\bigg)+2\bigg(\int_{t_n}^{t_{n+1}} (g(\Bu^n)-g(\Bu(s))dW(s), {e}_\Bu^{n}\bigg). \notag
	\end{align}
	
	Following a similar fashion as \eqref{eq:exi}, we need to balance the first term on the right-hand side of \eqref{eq:e5}.  Subtracting the trivial identity $0$ from both sides of \eqref{eq:SAV4}, and multiplying both sides of the resulted equation by $2e_\eta^{n+1}$ result in
	\begin{align}\label{eq:e6}
		&|e_\eta^{n+1}|^2-|e_\eta^n|^2+|e_\eta^{n+1}-e_\eta^n|^2=-2e_\eta^{n+1}(g(\Bu^n)\delta W^n, \tilde{\Bu}^{n+1}-\Bu^n-g(\Bu^n)\delta W^n)\\
		&=-2e_\eta^{n+1}\bigg(g(\Bu^n)\delta W^n, \tilde{e}_\Bu^{n+1}-e_\Bu^n-\int_{t_n}^{t_{n+1}}[g(\Bu^n)-g(\Bu(s))]dW(s)\bigg)\notag\\
		&\quad-2e_\eta^{n+1}\bigg(g(\Bu^n)\delta W^n, \Bu(t_{n+1})-\Bu(t_n)-\int_{t_n}^{t_{n+1}}g(\Bu(s))dW(s)\bigg).\notag
	\end{align}
	
	Note that the first term of right-hand side of \eqref{eq:e5} cancels with that of  \eqref{eq:e6}.  
	Thus, summing \eqref{eq:e5} and \eqref{eq:e6}, and applying the elementary inequality $2ab\leq \frac{1}{2}a^2+2b^2$ give
	\begin{align}\label{eq:e8}
		&|e_\eta^{n+1}|^2-|e_\eta^n|^2+|e_\eta^{n+1}-e_\eta^n|^2+2\bigg(\int_{t_n}^{t_{n+1}} \eta^{n+1}g(\Bu^n)-\eta(s){ g}(\Bu(s))dW(s), \tilde{e}_\Bu^{n+1}\bigg)\notag\\
		&=2\bigg(e_\eta^{n+1}g(\Bu^n)\delta W^n, {e}_\Bu^{n}+\int_{t_n}^{t_{n+1}}[g(\Bu^n)-g(\Bu(s))]dW(s)\bigg)\notag\\
		&\quad+2\bigg(\int_{t_n}^{t_{n+1}} (g(\Bu^n)-g(\Bu(s)))dW(s), \tilde{e}_\Bu^{n+1}-e_\Bu^n\bigg)
		+2\bigg(\int_{t_n}^{t_{n+1}} (g(\Bu^n)-g(\Bu(s))dW(s), {e}_\Bu^{n}\bigg)\notag\\
		&\quad-2e_\eta^{n+1}\bigg(g(\Bu^n)\delta W^n, \Bu(t_{n+1})-\Bu(t_n)-\int_{t_n}^{t_{n+1}}g(\Bu(s))dW(s)\bigg)\notag \\
		&\leq \frac{1}{2}\big|e_\eta^{n+1}-e_\eta^n|^2+4\bigg| \bigg(g(\Bu^n)\delta W^n, {e}_\Bu^{n}+\int_{t_n}^{t_{n+1}}[g(\Bu^n)-g(\Bu(s))]dW(s)\bigg)\bigg|^2+2\big(e_\eta^ng(\Bu^n)\delta W^n, {e}_\Bu^{n}\big)\notag\\
		&\quad+2\bigg(e_\eta^ng(\Bu^n)\delta W^n, \int_{t_n}^{t_{n+1}}(g(\Bu^n)-g(\Bu(s)))dW(s)\bigg)
		+\frac{1}{2}\|\tilde{e}_\Bu^{n+1}-e_\Bu^n\|^2\notag\\
		&\quad
		+2\bigg\|\int_{t_n}^{t_{n+1}} (g(\Bu^n)-g(\Bu(s)))dW(s)\bigg\|^2+2\bigg(\int_{t_n}^{t_{n+1}} (g(\Bu^n)-g(\Bu(s))dW(s), {e}_\Bu^{n}\bigg)\notag\\
		&\quad+\frac{1}{2}\big|e_\eta^{n+1}-e_\eta^n|^2+2\bigg|\bigg(g(\Bu^n)\delta W^n, \Bu(t_{n+1})-\Bu(t_n)-\int_{t_n}^{t_{n+1}}g(\Bu(s))dW(s)\bigg)\bigg|^2\notag\\
		&\quad-   
		2\bigg(e_\eta^{n}g(\Bu^n)\delta W^n, \Bu(t_{n+1})-\Bu(t_n)-\int_{t_n}^{t_{n+1}}g(\Bu(s))dW(s)\bigg).
	\end{align}
	
	Combine the results of estimate of pressure \eqref{eq:eI1}-\eqref{eq:eI3}, estimate of convection term \eqref{eq:con}, estimate of diffusion term \eqref{eq:diff} and estimate of stochastic term \eqref{eq:e8}, and
	substituting them into \eqref{eq:e3}, we obtain
	\begin{align*}
		&\|{e}_\Bu^{n+1}\|^2-\|e_\Bu^n\|^2+\frac{1}{2}\|\tilde{e}_\Bu^{n+1}-e_\Bu^n\|^2+\tau^2\|\nabla e_p^{n+1}\|^2-\tau^2\|\nabla e_p^n\|^2\notag\\
		&\quad+|e_\xi^{n+1}|^2-|e_\xi^n|^2+|e_\xi^{n+1}-e_\xi^n|^2+|e_\eta^{n+1}|^2-|e_\eta^n|^2\notag\\
		&\leq C\tau\|\Bu(t_n)\|_1^2 \|e_\Bu^n\|^2+ { C}\tau\|\nabla p(t_{n+1})-\nabla p(t_n)\|^2+C\tau^2\|\Bu^n\|^2\notag\\
		&\quad+C\int_{t_n}^{t_{n+1}} \|\Bu(t_n)-\Bu(s)\|^2ds+C\tau \|\Bu(t_{n+1})-\Bu(t_n)\|_1^2\|\Bu^n\|^2\notag\\
		&\quad+2\int_{t_n}^{t_{n+1}} \|\nabla\Bu(s)-\nabla\Bu(t_{n+1})\|^2ds +4\bigg| \bigg(g(\Bu^n)\delta W^n, {e}_\Bu^{n}+\int_{t_n}^{t_{n+1}}[g(\Bu^n)-g(\Bu(s))]dW(s)\bigg)\bigg|^2\notag\\
	\end{align*}
	\begin{align}
		&\quad+2\big(e_\eta^n{ g}(\Bu^n)\delta W^n, {e}_\Bu^{n}\big)+2\bigg(e_\eta^ng(\Bu^n)\delta W^n, \int_{t_n}^{t_{n+1}}[g(\Bu^n)-g(\Bu(s))dW(s)]\bigg)\notag\\
		&\quad+2\bigg\|\int_{t_n}^{t_{n+1}} (g(\Bu^n)-g(\Bu(s)))dW(s)\bigg\|^2+2\bigg(\int_{t_n}^{t_{n+1}} (g(\Bu^n)-g(\Bu(s))dW(s), {e}_\Bu^{n}\bigg)\notag\\		
		&\quad+2\bigg|{ \bigg(g(\Bu^n)\delta W^n, \Bu(t_{n+1})-\Bu(t_n)-\int_{t_n}^{t_{n+1}}g(\Bu(s))dW(s)\bigg)}\bigg|^2\notag\\
		&\quad-2\bigg(e_\eta^{n}g(\Bu^n)\delta W^n, \Bu(t_{n+1})-\Bu(t_n)-\int_{t_n}^{t_{n+1}}g(\Bu(s))dW(s)\bigg)\notag\\
		&\quad+C\tau \|e_\Bu^n\|^2+C\tau |e_\xi^{n+1}|^2+\tau^3(\|\nabla e_p^{n+1}\|^2+\|\nabla e_p^n\|^2)+C\tau^2\|\nabla p(t_{n+1})-\nabla p(t_n)\|^2.
	\end{align}
	Since the terms in the last line are of higher order, we omit them henceforth to simplify the proof.
	
	Summing over $n$, we have
	\begin{align}
		&\|{e}_\Bu^{k+1}\|^2+\frac{1}{2}\sum\limits_{n=0}^k\|\tilde{e}_\Bu^{n+1}-e_\Bu^n\|^2+\tau^2\|\nabla e_p^{k+1}\|^2
		+|e_\xi^{k+1}|^2+\sum\limits_{n=0}^k|e_\xi^{n+1}-e_\xi^n|^2+|e_\eta^{k+1}|^2\notag\\
		&\leq C\tau\sum\limits_{n=0}^k(1+\|\Bu(t_n)\|_1^2) \|e_\Bu^n\|^2+ { C}\tau\sum\limits_{n=0}^k\|\nabla p(t_{n+1})-\nabla p(t_n)\|^2+C\tau^2\sum\limits_{n=0}^k\|\Bu^n\|^2\notag\\
		&\quad+C\sum\limits_{n=0}^k\int_{t_n}^{t_{n+1}} \|\Bu(t_n)-\Bu(s)\|^2ds+C\tau\sum\limits_{n=0}^k \|\Bu(t_{n+1})-\Bu(t_n)\|_1^2\|\Bu^n\|^2\notag\\
		&\quad+2\sum\limits_{n=0}^k\int_{t_n}^{t_{n+1}} \|\nabla\Bu(s)-\nabla\Bu(t_{n+1})\|^2ds +4\sum\limits_{n=0}^k\bigg| \bigg({ g}(\Bu^n)\delta W^n, {e}_\Bu^{n}+\int_{t_n}^{t_{n+1}}[g(\Bu^n)-g(\Bu(s))]dW(s)\bigg)\bigg|^2\notag\\
		&\quad+2\sum\limits_{n=0}^k\big(e_\eta^n{ g}(\Bu^n)\delta W^n, {e}_\Bu^{n}\big)
		+2\sum\limits_{n=0}^k\bigg(e_\eta^n{g}(\Bu^n)\delta W^n,    \int_{t_n}^{t_{n+1}}[g(\Bu^n)-g(\Bu(s))]dW(s)\bigg)\notag\\
		&\quad+2\sum\limits_{n=0}^k\bigg\|\int_{t_n}^{t_{n+1}} ({ g}(\Bu^n)-{ g}(\Bu(s)))dW(s)\bigg\|^2
		+2\sum\limits_{n=0}^k\bigg(\int_{t_n}^{t_{n+1}} ({ g}(\Bu^n)-{ g}(\Bu(s))dW(s), {e}_\Bu^{n}\bigg)\notag\\
		&\quad+2\sum\limits_{n=0}^k{ \bigg|{ \bigg(g(\Bu^n)\delta W^n, \Bu(t_{n+1})-\Bu(t_n)-\int_{t_n}^{t_{n+1}}g(\Bu(s))dW(s)\bigg)}\bigg|^2}\notag\\
		&\quad-2\sum\limits_{n=0}^k\bigg(e_\eta^n g(\Bu^n)\delta W^n, \Bu(t_{n+1})-\Bu(t_n)-\int_{t_n}^{t_{n+1}}g(\Bu(s){)}dW(s)\bigg)\notag\\
		&:=C\tau\sum\limits_{n=0}^k(1+\|\Bu(t_n)\|_1^2) \|e_\Bu^n\|^2+A.
	\end{align}
	
	Now,  using the discrete Gronwall's lemma and taking expectation, we obtain
	\begin{align}\label{eq:maineu}
		&\mathbf{E}\|{e}_\Bu^{k+1}\|^2+\mathbf{E}|e_\xi^{k+1}|^2+\mathbf{E}|e_\eta^{k+1}|^2\notag\\
		&\leq \mathbf{E}\bigg[\exp\bigg({ C}\tau\sum\limits_{n=0}^k(1+\|\Bu(t_n)\|_1^2)\bigg)A\bigg]\notag\\
		&\leq \sqrt{\mathbf{E}\exp\bigg({ C}\tau\sum\limits_{n=0}^k(1+\|\Bu(t_n)\|_1^2)\bigg)} \sqrt{\mathbf{E}A^2}\leq C\sqrt{\mathbf{E}A^2},
	\end{align}
	where we have used Lemma \ref{lem:eiu}. 
	
	It should be pointed out that our constants $C$ so far is independent of $\Bu$ and $p$. 
	\vskip .1in
	{\bf Part 2: Estimates after taking expectation}
	
	It remains to estimate $\mathbf{E}A^2$. Using the elementary inequality 
	\eqref{eq:elem1} again,
	we can have
	\begin{align}
		\mathbf{E} A^2&\leq 12 \mathbf{E} \bigg[ \left({ C}\tau\sum\limits_{n=0}^k\|\nabla p(t_{n+1})-\nabla p(t_n)\|^2\right)^2+\left(C\tau^2\sum\limits_{n=0}^k\|\Bu^n\|^2\right)^2\notag\\
		&\quad+ \left(C\sum\limits_{n=0}^k\int_{t_n}^{t_{n+1}} \|\Bu(t_n)-\Bu(s)\|^2ds\right)^2+\left(C\tau\sum\limits_{n=0}^k \|\Bu(t_{n+1})-\Bu(t_n)\|_1^2\|\Bu^n\|^2\right)^2\notag\\
		&\quad
		+\bigg(2\sum\limits_{n=0}^k\int_{t_n}^{t_{n+1}} \|\nabla\Bu(s)-\nabla\Bu(t_{n+1})\|^2ds\bigg)^2\notag\\
		&\quad +\bigg(4\sum\limits_{n=0}^k\bigg| \bigg(g(\Bu^n)\delta W^n, {e}_\Bu^{n}+\int_{t_n}^{t_{n+1}}[g(\Bu^n)-g(\Bu(s))]dW(s)\bigg)\bigg|^2\bigg)^2\notag\\
		&\quad+\bigg(2\sum\limits_{n=0}^k\big(e_\eta^n g(\Bu^n)\delta W^n, {e}_\Bu^{n}\big)\bigg)^2
		+\bigg(2\sum\limits_{n=0}^k\bigg(e_\eta^ng(\Bu^n)\delta W^n,    \int_{t_n}^{t_{n+1}}[g(\Bu^n)-g(\Bu(s))]dW(s)\bigg)\bigg)^2\notag\\
		&\quad+\bigg(2\sum\limits_{n=0}^k\bigg\|\int_{t_n}^{t_{n+1}} (g(\Bu^n)-g(\Bu(s)))dW(s)\bigg\|^2\bigg)^2
		+\bigg(2\sum\limits_{n=0}^k\bigg(\int_{t_n}^{t_{n+1}} (g(\Bu^n)-g(\Bu(s))dW(s), {e}_\Bu^{n}\bigg)\bigg)^2\notag\\
		&\quad+\bigg(2\sum\limits_{n=0}^k\bigg|{ \bigg(g(\Bu^n)\delta W^n, \Bu(t_{n+1})-\Bu(t_n)-\int_{t_n}^{t_{n+1}}g(\Bu(s))dW(s)\bigg)}\bigg|^2\bigg)^2\notag\\
		&\quad+\bigg(2\sum\limits_{n=0}^k\bigg(e_\eta^n g(\Bu^n)\delta W^n, \Bu(t_{n+1})-\Bu(t_n)-\int_{t_n}^{t_{n+1}}g(\Bu(s){ )}dW(s)\bigg)\bigg)^2\bigg]\notag\\
		&:=12\sum\limits_{k=1}^{12}\mathbf{E} A_{k}. 
	\end{align}
	Now, we estimate these $12$ terms one by one.
	
	The estimates of $A_1$-$A_5$ are simple, and we tackle them together. 
	An application of Assumption \ref{assump:3} and Lemma \ref{lem:sta}, and \eqref{eq:conu} leads to
	\begin{align}
		\mathbf{E} A_{1}&=\mathbf{E}\left({ C}\tau\sum\limits_{n=0}^k\|\nabla p(t_{n+1})-\nabla p(t_n)\|^2\right)^2
		\leq { C}T\tau\sum\limits_{n=0}^k \mathbf{E}\|\nabla p(t_{n+1})-\nabla p(t_n)\|^4
		\leq C\tau^{2\theta};  \\
		\mathbf{E}A_2&\leq C\tau^3\mathbf{E} \sum\limits_{n=0}^k\|\Bu^n\|^4\leq C\tau^2;  \\
		\mathbf{E} A_3&\leq C\sum\limits_{n=0}^k \int_{t_n}^{t_{n+1}} \mathbf{E}\|\Bu(t_n)-\Bu(s)\|^4ds\leq C\tau^{4\theta}; \\
		\mathbf{E}A_4&\leq C\tau\sum\limits_{n=0}^k \mathbf{E} \|\Bu(t_{n+1})-\Bu(t_n)\|_1^4\|\Bu^n\|^4
		\leq C\tau\sum\limits_{n=0}^k \sqrt{\mathbf{E} \|\Bu(t_{n+1})-\Bu(t_n)\|_1^8}\sqrt{\mathbf{E}\|\Bu^n\|^8} \notag\\
		&\leq C\tau^{4\theta};  \\
		\mathbf{E}A_5&\leq C\sum\limits_{n=0}^k \int_{t_n}^{t_{n+1}} \mathbf{E}\|\nabla \Bu(s)-\nabla\Bu(t_{n+1})\|^4ds\leq C\tau^{4\theta}. 
	\end{align}
	By standard BDG inequality and Assumption \ref{assump:3} and Corollary \ref{cor:up},
	\begin{align}
		\mathbf{E}A_6&=\mathbf{E} \bigg(4\sum\limits_{n=0}^k\bigg| \bigg(g(\Bu^n)\delta W^n, {e}_\Bu^{n}+\int_{t_n}^{t_{n+1}}[g(\Bu^n)-g(\Bu(s))]dW(s)\bigg)\bigg|^2\bigg)^2\notag\\
		&\leq 64\mathbf{E}\bigg(\sum\limits_{n=0}^k\big| \big(g(\Bu^n)\delta W^n, {e}_\Bu^{n}\big)\big|^2\bigg)^2+64\mathbf{E}\bigg(\sum\limits_{n=0}^k\bigg| \bigg(g(\Bu^n)\delta W^n, \int_{t_n}^{t_{n+1}}[g(\Bu^n)-g(\Bu(s))]dW(s)\bigg)\bigg|^2\bigg)^2\notag\\
		&=64\sum\limits_{n=0}^k\mathbf{E} |(g(\Bu^n)\delta W^n, {e}_\Bu^{n}\big)|^4+128\sum\limits_{i\neq j}^k \mathbf{E} |(g(\Bu^i)\delta W^i, {e}_\Bu^{i}\big)|^2
		\mathbf{E}|(g(\Bu^j)\delta W^j, {e}_\Bu^{j}\big)|^2\notag\\
		&\quad+64k\sum\limits_{n=0}^k \mathbf{E}\bigg| \bigg(\int_{t_n}^{t_{n+1}}g(\Bu^n) dW(s), \int_{t_n}^{t_{n+1}}[g(\Bu^n)-g(\Bu(s))]dW(s)\bigg)\bigg|^4\notag\\
		&\leq C\tau^2\sum\limits_{n=0}^k \mathbf{E}\|e_{\Bu}^n\|^4+C\sum\limits_{i\neq j}^k \tau^2 \mathbf{E}\|e_\Bu^i\|^2 \mathbf{E} \|e_\Bu^j\|^2\notag\\
		&\quad+64k\sum\limits_{n=0}^k \mathbf{E}\bigg[\bigg\|\int_{t_n}^{t_{n+1}}g(\Bu^n)dW(s)\bigg\|^4\bigg\|\int_{t_n}^{t_{n+1}}[g(\Bu^n)-g(\Bu(s))]dW(s)\bigg\|^4\bigg]\notag\\
		&\leq C\tau^2\sum\limits_{n=0}^k (\mathbf{E} \|\Bu(t_n)\|^4+\mathbf{E}\|\Bu^n\|^4)+C\sum\limits_{i\neq j}^k \tau^2 \mathbf{E}\|e_\Bu^i\|^2 \mathbf{E} \|e_\Bu^j\|^2\notag\\
		&\quad+64k\sum\limits_{n=0}^k \sqrt{\mathbf{E}\bigg\|\int_{t_n}^{t_{n+1}}g(\Bu^n)dW(s)\bigg\|^8}\sqrt{\mathbf{E}\bigg\|\int_{t_n}^{t_{n+1}}[g(\Bu^n)-g(\Bu(s))]dW(s)\bigg\|^8}\notag\\
		&\leq C\tau^2\sum\limits_{n=0}^k (\mathbf{E} \|\Bu(t_n)\|^4+\mathbf{E}\|\Bu^n\|^4)+C\sum\limits_{i\neq j}^k \tau^2 \mathbf{E}\|e_\Bu^i\|^2 \mathbf{E} \|e_\Bu^j\|^2+C\tau^2\notag\\
		&\leq C\tau+C\bigg(\tau\sum\limits_{n=0}^k \mathbf{E} \|e_\Bu^n\|^2\bigg)^2.
	\end{align}
	A simple application of the Young's inequality implies the bound of $A_7$. 
	\begin{align}
		\mathbf{E}A_7&=\sum\limits_{n=0}^k\mathbf{E} [e_\eta^n(g(\Bu^n)\delta W^n, e_\Bu^n)]^2
		=\sum\limits_{n=0}^k \mathbf{E}|e_\eta^n|^2 \mathbf{E}|(g(\Bu^n)\delta W^n, e_\Bu^n)|^2
		\leq C_g^2\tau\sum\limits_{n=0}^k \mathbf{E}|e_\eta^n|^2 \mathbf{E}\|e_\Bu^n\|^2\notag\\
		&\leq \frac{1}{4} \big[\max\limits_{0\leq n \leq k} \mathbf{E}\|e_\Bu^n\|^2\big]^2+C\bigg(\tau \sum\limits_{n=0}^k \mathbf{E}|e_\eta^n|^2\bigg)^2.
	\end{align}
	\begin{align}
		\text{Similarly,} \;\; 
		\mathbf{E}A_8&=4\mathbf{E}\bigg(\sum\limits_{n=0}^k\bigg(e_\eta^ng(\Bu^n)\delta W^n,    \int_{t_n}^{t_{n+1}}[g(\Bu^n)-g(\Bu(s))]dW(s)\bigg)\bigg)^2\notag\\
		&=4\mathbf{E}\sum\limits_{n=0}^k\bigg(e_\eta^ng(\Bu^n)\delta W^n,    \int_{t_n}^{t_{n+1}}[g(\Bu^n)-g(\Bu(s))]dW(s)\bigg)^2\notag\\
		&\quad+8\sum\limits_{i\neq j} \mathbf{E}\bigg(e_\eta^ig(\Bu^i)\delta W^i,    \int_{t_i}^{t_{i+1}}[g(\Bu^i)-g(\Bu(s))]dW(s)\bigg)\notag\\
		&\quad\times\mathbf{E}\bigg(e_\eta^jg(\Bu^j)\delta W^j,    \int_{t_j}^{t_{j+1}}[g(\Bu^j)-g(\Bu(s))]dW(s)\bigg)\hskip 2in\notag\\
		&:=A_{81}+A_{82}.
	\end{align}
	By Assumption \ref{assump:1}, we easily have
	\begin{align}
		A_{81}
		&\leq 4\sum\limits_{n=0}^k \mathbf{E}|e_\eta^n|^2 \mathbf{E}\bigg[\|g(\Bu^n)\delta W^n\|^2\bigg\|\int_{t_n}^{t_{n+1}}[g(\Bu^n)-g(\Bu(s))]dW(s)\bigg\|^2\bigg]\notag\\
		&\leq 4 \sum\limits_{n=0}^k \mathbf{E}|e_\eta^n|^2\sqrt{\mathbf{E}\|g(\Bu^n)\delta W^n\|^4}\sqrt{\mathbf{E}\bigg\|\int_{t_n}^{t_{n+1}}[g(\Bu^n)-g(\Bu(s))]dW(s)\bigg\|^4}\notag\\
		&\leq C\tau^2 \sum\limits_{n=0}^k \mathbf{E}|e_\eta^n|^2\leq C\tau^2+C\bigg(\tau\sum\limits_{n=0}^k\mathbf{E}|e_\eta^n|^2\bigg)^2;
	\end{align}
	and by Ito isometry \cite[Proposition 2.10]{Kruse14} and Assumption \ref{assump:1}, 
	\begin{align}
		A_{82}&\leq 8\bigg[\sum\limits_{n=0}^k \mathbf{E}\bigg(e_\eta^ng(\Bu^n)\delta W^n,    \int_{t_n}^{t_{n+1}}[g(\Bu^n)-g(\Bu(s))]dW(s)\bigg) \bigg]^2\notag\\
		&=8\bigg[\sum\limits_{n=0}^k \mathbf{E} |e_\eta^n| \mathbf{E}\bigg(g(\Bu^n)\delta W^n,    \int_{t_n}^{t_{n+1}}[g(\Bu^n)-g(\Bu(s))]dW(s)\bigg) \bigg]^2\notag\\
		&\leq 8\bigg[\sum\limits_{n=0}^k \mathbf{E} |e_\eta^n| \mathbf{E} \int_{t_n}^{t_{n+1}}Tr(g(\Bu^n)Q^{1/2} ((g(\Bu^n)-g(\Bu(s))Q^{1/2}))^*)ds \bigg]^2\notag\\
		&\leq C\bigg(\tau\sum\limits_{n=0}^k \mathbf{E} |e_\eta^n|^2\bigg)^2.
	\end{align}
	Hence,
	\begin{align}
		\mathbf{E}A_8&\leq C\tau^2+C\bigg(\tau\sum\limits_{n=0}^k \mathbf{E} |e_\eta^n|^2\bigg)^2.
	\end{align}
	Similarly, by the standard BDG inequality (Lemma \ref{lem:BDG}) and \eqref{eq:rem1}, we have
	\begin{align}
		\mathbf{E}A_9&=4\mathbf{E} \bigg(\sum\limits_{n=0}^k\bigg\|\int_{t_n}^{t_{n+1}} ({ g}(\Bu^n)-{g}(\Bu(s)))dW(s)\bigg\|^2\bigg)^2\notag\\
		&=4\sum\limits_{n=0}^k \mathbf{E} \bigg\|\int_{t_n}^{t_{n+1}} ({ g}(\Bu^n)-{ g}(\Bu(s)))dW(s)\bigg\|^4\notag\\
		&\quad+8\sum\limits_{i,j=0, i\neq j}^k \mathbf{E} \bigg\|\int_{t_i}^{t_{i+1}} ({ g}(\Bu^i)-{ g}(\Bu(s)))dW(s)\bigg\|^2 \mathbf{E} \bigg\|\int_{t_j}^{t_{j+1}} ({g}(\Bu^j)-{ g}(\Bu(s)))dW(s)\bigg\|^2\notag\\
		&\leq C\sum\limits_{n=0}^k \bigg(\int_{t_n}^{t_{n+1}}\mathbf{E}\|g(\Bu^n)-g(\Bu(s))\|_{\mathcal{L}_2^0}^2ds \bigg)^2\notag\\
		&\quad+C\sum\limits_{i,j=0, i\neq j}^k \int_{t_i}^{t_{i+1}} \mathbf{E} \|{ g}(\Bu^i)-{ g}(\Bu(s))\|_{\mathcal{L}_2^0}^2ds \int_{t_j}^{t_{j+1}} \mathbf{E} \|{ g}(\Bu^j)-{ g}(\Bu(s))\|_{\mathcal{L}_2^0}^2ds\notag\\
		&\leq C\sum\limits_{n=0}^k \tau^2 (\mathbf{E}\|e_\Bu^n\|^2+\tau^{2\theta})^2+C\tau^2\sum\limits_{i,j=0, i\neq j}^k (\mathbf{E}\|e_\Bu^i\|^2+\tau^{2\theta})
		(\mathbf{E}\|e_\Bu^j\|^2+\tau^{2\theta})\notag\\
		&\leq C\bigg(\tau\sum\limits_{n=0}^k  (\mathbf{E}\|e_\Bu^n\|^2+\tau^{2\theta})^2\bigg)^2.
	\end{align}
	Using  Assumption \ref{assump:3} and Ito isometry again, we have
	\begin{align}
		\mathbf{E}A_{10}&=\mathbf{E}\bigg(2\sum\limits_{n=0}^k\bigg(\int_{t_n}^{t_{n+1}} (g(\Bu^n)-g(\Bu(s))dW(s), {e}_\Bu^{n}\bigg)\bigg)^2\notag\\
		&=4\sum\limits_{n=0}^k\mathbf{E}\bigg(\int_{t_n}^{t_{n+1}} (g(\Bu^n)-g(\Bu(s))dW(s), {e}_\Bu^{n}\bigg)^2\notag\\
		&\leq 4\sum\limits_{n=0}^k \mathbf{E}\bigg\|\int_{t_n}^{t_{n+1}} (g(\Bu^n)-g(\Bu(s))dW(s)\bigg\|^2 \mathbf{E}\|e_\Bu^n\|^2\notag\\
		&\leq C\sum\limits_{n=0}^k\int_{t_n}^{t_{n+1}} \mathbf{E}\|g(\Bu^n)-g(\Bu(s))\|_{\mathcal{L}_2^0}^2ds \mathbf{E}\|e_\Bu^n\|^2\notag\\
		&\leq C\sum\limits_{n=0}^k\tau(\mathbf{E}\|e_\Bu^n\|^2+\tau^{2\theta}) \mathbf{E}\|e_\Bu^n\|^2\notag\\
		&\leq \frac{1}{4}\big[\max\limits_{0\leq n\leq k}\mathbf{E}\|e_\Bu^n\|^2\big]^2+C\bigg(\sum\limits_{n=0}^k\tau(\mathbf{E}\|e_\Bu^n\|^2+\tau^{2\theta})\bigg)^2.
	\end{align}
	Before giving an estimate of $A_{11}$, let us find a continuity result of $\Bu$. By \eqref{eq:conu} and standard BDG inequallity, we can get
	\begin{align}
		&\mathbf{E}\bigg\|\Bu(t_{n+1})-\Bu(t_n)-\int_{t_n}^{t_{n+1}}g(\Bu(s))dW(s)\bigg\|^8 \notag\\
		&\leq C\mathbf{E}\|\Bu(t_{n+1})-\Bu(t_n)\|^8+C\mathbf{E}\bigg\|\int_{t_n}^{t_{n+1}}g(\Bu(s))dW(s)\bigg\|^8\notag\\
		&\leq C\tau^{8\theta}+C\mathbf{E}\bigg(\int_{t_n}^{t_{n+1}}\|g(\Bu(s))\|_{\mathcal{L}_2^0}^2ds\bigg)^{4}\leq C\tau^{8\theta}+C\tau^4.
	\end{align}
	Therefore,
	\begin{align}
		\mathbf{E}A_{11} &=\mathbf{E}\bigg({ 2}\sum\limits_{n=0}^k\bigg|{ \bigg(g(\Bu^n)\delta W^n, \Bu(t_{n+1})-\Bu(t_n)-\int_{t_n}^{t_{n+1}}g(\Bu(s))dW(s)\bigg)}\bigg|^2\bigg)^2\notag\\
		&\leq { 4(k+1)}\sum\limits_{n=0}^k\mathbf{E}\bigg[ \|g(\Bu^n)\delta W^n\|^4 \bigg\|\Bu(t_{n+1})-\Bu(t_n)-\int_{t_n}^{t_{n+1}}g(\Bu(s))dW(s)\bigg\|^4\bigg]\notag\\
		&\leq { 4(k+1)}\sum\limits_{n=0}^k \sqrt{\mathbf{E} \|g(\Bu^n)\delta W^n\|^8} \sqrt{\mathbf{E}\bigg\|\Bu(t_{n+1})-\Bu(t_n)-\int_{t_n}^{t_{n+1}}g(\Bu(s))dW(s)\bigg\|^8}\notag\\
		&\leq C\tau^{4\theta}+C\tau^2. 
	\end{align}
	
	Since the true solution satisfies
	\begin{align}\label{eq:ut1}
		\Bu(t_{n+1})-\Bu(t_n)=\int_{t_n}^{t_{n+1}}\Delta\Bu(s)ds+\int_{t_n}^{t_{n+1}} [\Bv(s)\cdot\nabla]\Bu(s)ds+\int_{t_n}^{t_{n+1}}g(\Bu(s))dW(s),
	\end{align}
	applying Assumption \ref{assump:1} and Cauchy-Schwartz inequality, we obtain
	\begin{align}\label{eq:a12}
		\mathbf{E}A_{12}&=4\mathbf{E}\bigg(\sum\limits_{n=0}^k{ \bigg(}e_\eta^{n}g(\Bu^n)\delta W^n, \Bu(t_{n+1})-\Bu(t_n)-\int_{t_n}^{t_{n+1}}g(\Bu(s))dW(s){ \bigg)} \bigg)^2\notag\\
		&=4\mathbf{E}\bigg(\sum\limits_{n=0}^k{ \bigg(}e_\eta^{n}g(\Bu^n)\delta W^n, \int_{t_n}^{t_{n+1}}\Delta\Bu(s)ds+\int_{t_n}^{t_{n+1}} [\Bv(s)\cdot\nabla]\Bu(s)ds { \bigg)}\bigg)^2\notag\\
		&=4\mathbf{E}\bigg(\sum\limits_{n=0}^k{ \bigg(}e_\eta^{n}(-\Delta)^{\frac{1}{2}}g(\Bu^n)\delta W^n, \int_{t_n}^{t_{n+1}}(-\Delta)^{\frac{1}{2}}\Bu(s)ds+\int_{t_n}^{t_{n+1}} (-\Delta)^{-\frac{1}{2}}[\Bv(s)\cdot\nabla]\Bu(s)ds{ \bigg)} \bigg)^2\notag\\
		&\leq 4\mathbf{E} \Bigg[\sum\limits_{n=0}^k \|e_\eta^{n}(-\Delta)^{\frac{1}{2}}g(\Bu^n)\delta W^n\|^2  \sum\limits_{n=0}^k \bigg\|\int_{t_n}^{t_{n+1}}(-\Delta)^{\frac{1}{2}}\Bu(s)ds+\int_{t_n}^{t_{n+1}} (-\Delta)^{-\frac{1}{2}}[\Bv(s)\cdot\nabla]\Bu(s)ds \bigg\|^2\bigg]\notag\\
		&\leq 2\mathbf{E} \bigg(\sum\limits_{n=0}^k \|e_\eta^{n}(-\Delta)^{\frac{1}{2}}g(\Bu^n)\delta W^n\|^2\bigg)^2\notag\\
		&\quad+2\mathbf{E}\bigg(\sum\limits_{n=0}^k \bigg\|\int_{t_n}^{t_{n+1}}(-\Delta)^{\frac{1}{2}}\Bu(s)ds+\int_{t_n}^{t_{n+1}} (-\Delta)^{-\frac{1}{2}}[\Bv(s)\cdot\nabla]\Bu(s)ds \bigg\|^2\bigg)^2\notag\\
		&:={ \mathbf{E}A_{121}+\mathbf{E}A_{122}}.
	\end{align}
	Expanding the summation and applying our stability result from Lemma \ref{lem:sta} and Assumpton \ref{assump:1}, we arrive at
	\begin{align}
		\mathbf{E}A_{121}&=2\sum\limits_{n=0}^k \mathbf{E}\|e_\eta^n{ (-\Delta)^{\frac{1}{2}}}	
		g(\Bu^n)\delta W^n\|^4+2\sum\limits_{i\neq j, i,j=0}^k\mathbf{E}\|e_\eta^i{ (-\Delta)^{\frac{1}{2}}}g(\Bu^i)\delta W^i\|^2\mathbf{E}\|e_\eta^j{ (-\Delta)^{\frac{1}{2}}}g(\Bu^j)\delta W^j\|^2\notag\\
		&\leq C\sum\limits_{n=0}^k\mathbf{E}|e_\eta^n|^4\tau^2+{ C}\bigg(\tau\sum\limits_{n=0}^k\mathbf{E}|e_\eta^n|^2\bigg)^2\notag\\
		&\leq C\sum\limits_{n=0}^k \mathbf{E}(1+|\eta^n|^4)\tau^2+{ C}\bigg(\tau\sum\limits_{n=0}^k\mathbf{E}|e_\eta^n|^2\bigg)^2\notag\\
		&\leq C\tau+{ C}\bigg(\tau\sum\limits_{n=0}^k\mathbf{E}|e_\eta^n|^2\bigg)^2.
	\end{align}
	The estimate of $\mathbf{E}A_{122}$ requires the following estimate (cf. \cite{CarelliP12,Giga85}):
	$$\|(-\Delta)^{-\frac{1}{2}}P_{\mathbf{H}}[\Bv\cdot\nabla]\Bu\|\leq C\|(-\Delta)^{\frac{1}{2}}\Bv\|\|(-\Delta)^{\frac{1}{2}}\Bu\|.$$
	Hence, by Lemma \ref{lem:eiu} and assumption on $v$, 
	\begin{align*}
		\mathbf{E}A_{122}&=2\mathbf{E}\bigg(\sum\limits_{n=0}^k \bigg\|\int_{t_n}^{t_{n+1}}(-\Delta)^{\frac{1}{2}}\Bu(s)ds+\int_{t_n}^{t_{n+1}} (-\Delta)^{-\frac{1}{2}}[\Bv(s)\cdot\nabla]\Bu(s)ds \bigg\|^2\bigg)^2\notag\\
		&\leq 8\mathbf{E}\bigg(\sum\limits_{n=0}^k\bigg\|\int_{t_n}^{t_{n+1}}(-\Delta)^{\frac{1}{2}}\Bu(s)ds\bigg\|^2+\bigg\|\int_{t_n}^{t_{n+1}} (-\Delta)^{-\frac{1}{2}}[\Bv(s)\cdot\nabla]\Bu(s)ds\bigg\|^2\bigg)^2\notag\\
		&\leq 8\mathbf{E}\bigg(\tau\sum\limits_{n=0}^k\int_{t_n}^{t_{n+1}}(\|(-\Delta)^{\frac{1}{2}}\Bu(s)\|^2+{ C}\|(-\Delta)^{\frac{1}{2}}\Bv(s)\|^2\|(-\Delta)^{\frac{1}{2}}\Bu(s)\|^2)ds\bigg)^2\notag\\
	\end{align*}
	\begin{align}
		&\leq 8\tau^2\mathbf{E}\bigg(\int_0^{t_{k+1}} (\|(-\Delta)^{\frac{1}{2}}\Bu(s)\|^2ds+{ C}\|(-\Delta)^{\frac{1}{2}}\Bu(s)\|^2\|(-\Delta)^{\frac{1}{2}}\Bv(s)\|^2)ds\bigg)^2\notag\\
		&\leq 8\tau^2 \mathbf{E}[\sup\limits_t({ C}\|(-\Delta)^{\frac{1}{2}}\Bv(t)\|^2+1)]\mathbf{E}\bigg(\int_0^{t_{k+1}}\|(-\Delta)^{\frac{1}{2}}\Bu(s)\|^2\bigg)^2\notag\\
		&\leq 8\tau^2 \mathbf{E}[\sup\limits_t(\|(-\Delta)^{\frac{1}{2}}\Bv(t)\|^2+1)]\mathbf{E}\exp\bigg(\int_0^{t_{k+1}}\|(-\Delta)^{\frac{1}{2}}\Bu(s)\|^2ds\bigg)\notag\\
		&\leq C\tau^2. 
	\end{align}
	
	We collect estimates of {$A_1$-$A_{12}$}, and substitute them into \eqref{eq:maineu} to have
	\begin{align}
		&\mathbf{E}\|{e}_\Bu^{k+1}\|^2+\mathbf{E}|e_\xi^{k+1}|^2+\mathbf{E}|e_\eta^{k+1}|^2\notag\\
		&\leq C\bigg(\tau^{2\theta}+\tau\sum\limits_{n=0}^k(\mathbf{E}\|e_\Bu^n\|^2+\mathbf{E}|e_\eta^n|^2)\bigg)+\frac{1}{2}\max\limits_{0\leq n\leq k}\mathbf{E}\|e_\Bu^n\|^2.
	\end{align}
	Taking $\max$  on the left hand side and cancelling similar terms on the right hand side, we obtain
	\begin{align}
		&\max\limits_{0\leq n\leq k}[\mathbf{E}\|{e}_\Bu^{n+1}\|^2+\mathbf{E}|e_\xi^{n+1}|^2+\mathbf{E}|e_\eta^{n+1}|^2]\notag\\
		&\leq C\bigg(\tau^{2\theta}+\tau\sum\limits_{n=0}^k(\mathbf{E}\|e_\Bu^n\|^2+\mathbf{E}|e_\eta^n|^2)\bigg).
	\end{align}
	Then, the discrete Gronwall's inequality implies the result. 
\end{proof}
\begin{rem}
	From the estimate of $A_{12}$, it is evident that the mean-reverting term $g(\Bu^n)\delta W^n$ in \eqref{eq:SAV4} is crucial, as it allows for the derivation of the first two lines of \eqref{eq:a12} using \eqref{eq:ut1}. See also Remark \ref{rem1}. 
\end{rem}

\section{Numerical experiments}
In this section, two numerical examples are provided for illustration of our main theoretical result \eqref{eq:final}. 

\subsection{Accuracy test on homogeneous Dirichlet boundary condition} In this test, we shall freeze $T=0.2$, and take $g(\Bu)=I$ and $\text{diag}(2I-{\cos(\Bu)})$ respectively in \eqref{model:NS}. The initial condition is chosen as
\begin{equation}
	\Bu_0({ \boldsymbol{x}})=
	\begin{pmatrix}
		-128 x_1^2(x_1-1)^2 x_2(x_2-1)(2 x_2-1) \\
		128 x_2^2(x_2-1)^2 x_1(x_1-1)(2 x_1-1)
	\end{pmatrix}
\end{equation}
and the Q-Wiener process is approximated as
\begin{align}
	W(t)\approx\sum\limits_{j_1=1}^4\sum\limits_{j_2=1}^4 \frac{1}{(j_1+j_2)^{1+\frac{\epsilon}{2}}} \Be_{j_1j_2}(\Bx)\beta_{j_1j_2}(t),
\end{align}
with $\epsilon=1e-4$
and 
$
\Be_{j_1, j_2}(\Bx)=\left(\sin \left(j_1 \pi x_1\right) \sin \left(j_2 \pi x_2\right), \sin \left(j_1 \pi x_1\right) \sin \left(j_2 \pi x_2\right)\right)^T.
$

We employ the Fourier spectral method with $40$ modes in each spatial direction, ensuring that the spatial discretization error remains negligible. Since the exact solution is unavailable, we use the numerical solution with $\tau_0 = 1/12800$ as a reference surrogate. The temporal errors of the velocity and pressure at time $T$ are then approximated as
\begin{align}\label{numer:err}
	e_\Bu^n &:= \sqrt{\mathbf{E}\|\Bu(T,\cdot)-\Bu^n\|^2}\approx\sqrt{\frac{1}{300}\sum_{\ell=1}^{300} \|\Bu^n(\omega_\ell)-\Bu(T,\omega_\ell)\|^2}, \notag\\
	e_{p}^n &:= \left(\mathbf{E}\left[\left\|\int_0^T p(s)\mathrm{d}s - \tau \sum_{k=1}^{n} p^k\right\|^2\right]\right)^{1/2}
	\approx\left(\frac{1}{300} \sum_{\ell=1}^{300}\left\|\tau_0 \sum_{k=1}^{T/\tau_0} p(k\tau_0,\omega_\ell)-\tau\sum_{k=1}^{T/\tau} p^k(\omega_\ell)\right\|^2\right)^{1/2},
\end{align}
where $T = n\tau$. The definition of $e_p^n$ follows the measurement introduced in \cite{FengV22}. Under these settings, the numerical results for both choices of $g(\Bu)$ are summarized in Table~\ref{table1}.

It is evident that, in both cases, the temporal convergence rate of the velocity slightly surpasses $\mathcal{O}(\sqrt{\tau})$ whereas the pressure error demonstrates a convergence rate slightly below, or at best comparable to $\mathcal{O}(\sqrt{\tau})$. This behavior may stem from the inaccuracy of our so-called ``exact" solution and its resonance with numerical solutions computed at different step sizes $\tau$.  
\renewcommand{\arraystretch}{1.5}
\begin{table}
	\centering
	\begin{tabular}{c|c|c|c|c|c|c|c|c}
		\hline \multirow{2}{*}{$\tau$} & \multicolumn{4}{|c|}{$g(\Bu)=1$ } & \multicolumn{4}{c}{{ $g(\Bu)=2-\cos(\Bu)$} } \\
		\cline { 2 - 9 } & $e_\Bu^n$ & order &$e_p^n$ &order& $e_\Bu^n$ & order &$e_p^n$& order\\
		\hline $\frac{1}{200}$  & { $1.16\mathrm{e}-02$}  & { -} & $6.18\mathrm{e}-02$ & { -} &    $1.21\mathrm{e}-02$   & { -} & $6.23\mathrm{e}-02$ & { -} \\ 
		$\frac{1}{400}$  &   $7.60 \mathrm{e}-03$  &  0.61 &$4.24\mathrm{e}-02$ & 0.54 &   $7.67\mathrm{e}-03$  & { 0.66} & $4.26\mathrm{e}-02$& {0.55}\\ 
		$\frac{1}{800}$  &  $4.97 \mathrm{e}-03$  &  0.61  &$2.99\mathrm{e}-02$ & 0.50&  $4.86\mathrm{e}-03$ &   { 0.66}  & $2.95\mathrm{e}-02$& { 0.53}\\ 
		$\frac{1}{1600}$  & $3.03 \mathrm{e}-03$  & 0.71  &$2.15\mathrm{e}-02$ & 0.48&  $3.09\mathrm{e}-03$ &   { 0.65}  & $2.14\mathrm{e}-02$& {0.47}\\ 
		$\frac{1}{3200}$  &  $1.96 \mathrm{e}-03$  &  0.63 &$1.51\mathrm{e}-02$ & 0.51&  $1.98\mathrm{e}-03$ &   { 0.64}   &$1.49\mathrm{e}-02$ & {0.52}\\
		\hline
	\end{tabular}
	\captionsetup{width=0.9\hsize}
	\caption{Convergence order of $\Bu$ and $p$ of the proposed method for both choices of $g$.}\label{table1}
\end{table}

Additionally, both auxiliary processes, 
$\xi$ and $\eta$ evolve with mean 1 and small variances, as shown in Figs. \ref{fig3} and \ref{fig4}. For brevity, we present the results at 
$T=0.2$ with $
g(\Bu)=1$ for different time steps. Similar patterns are observed in the other case.
\begin{figure}[htbp]
	\centering
	\includegraphics[width=0.82\textwidth]{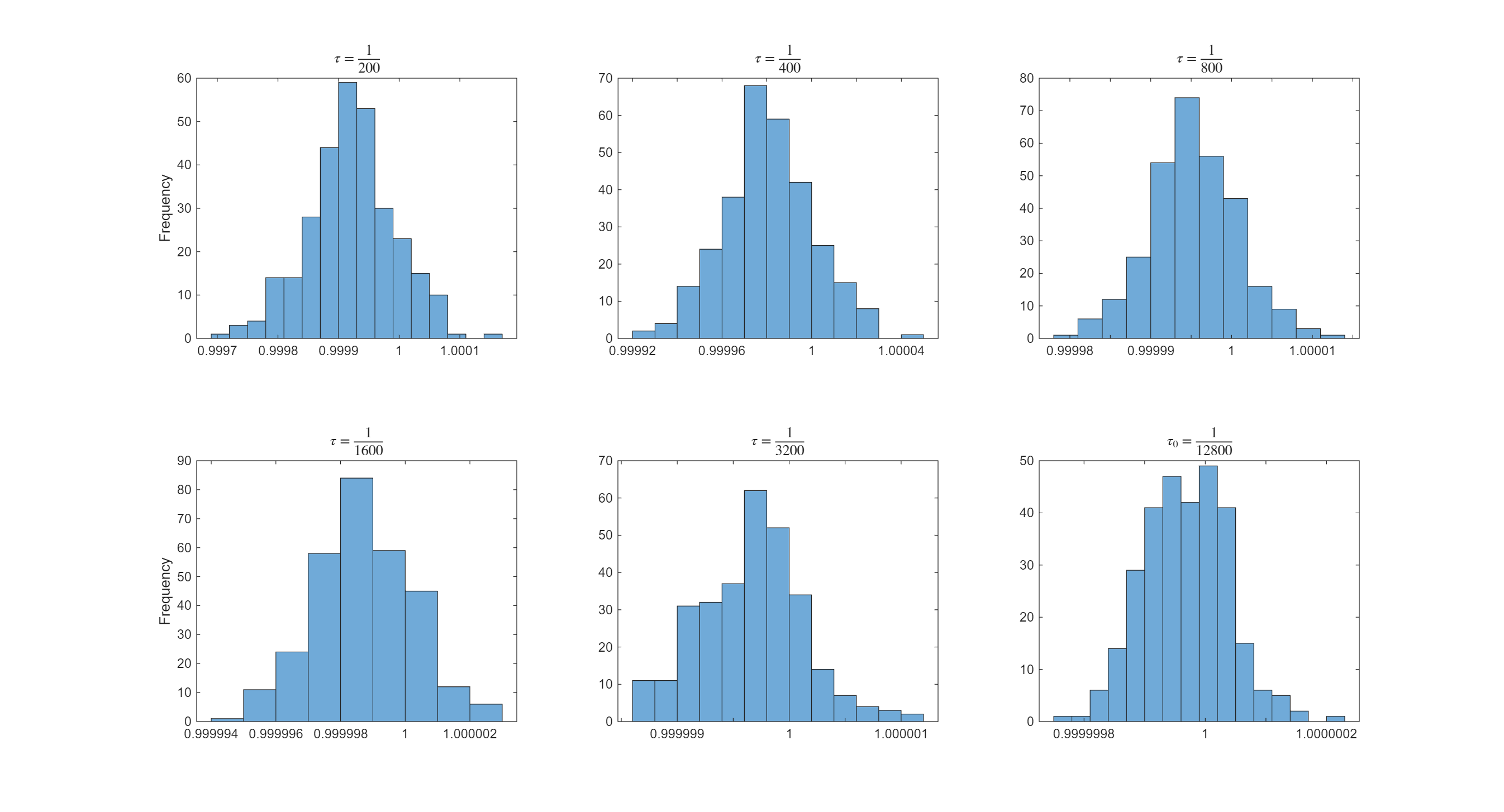}\vskip -0.2in
	\caption{Distribution of $\xi$ at $T$ with $g(\Bu)=1$ for various timesteps.} 
	\label{fig3}
\end{figure}

\begin{figure}[htbp]
	\centering
	\includegraphics[width=0.82\textwidth]{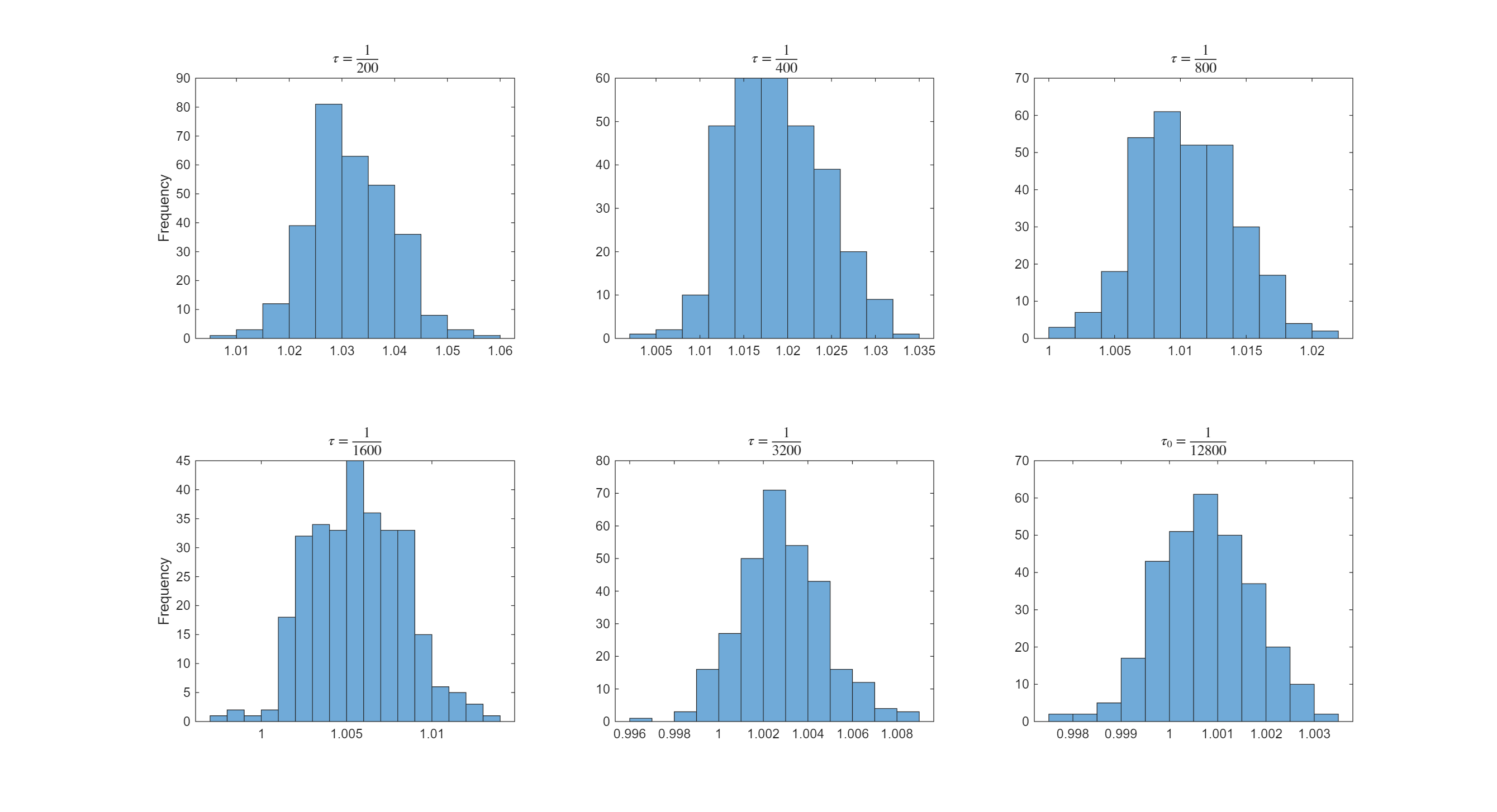}\vskip -0.2in
	\caption{Distribution of $\eta$ at $T$ with $g(\Bu)=1$ for various timesteps.} 
	\label{fig4}
\end{figure}

\subsection{Stochastic shear layer roll-up problem}
In this subsection, we study the shear layer roll-up flow perturbed by noise on the domain $[0,1]^2$, with doubly periodic-boundary condition and the following initial condition:
\begin{equation}
	\Bu_0({ \boldsymbol{x}})=
	\begin{pmatrix}
		\begin{cases}
			\tanh(30(x_2-0.25)), \ \ \text{if} \ x_2\leq 0.5\\
			\tanh(30(0.75-x_2)),\ \ \text{if}\ x_2>0.5,
		\end{cases}\\
		0.05\sin(2\pi x_1)
	\end{pmatrix}.
\end{equation}
It is worth noting that, in both our numerical analysis and the preceding numerical experiment, the effect of the Reynolds number was suppressed by fixing it to $1$. In the present test, however, we set the Reynolds number to $10^4$, which corresponds to replacing $\Delta \Bu$ with $(1\mathrm{e}{-4})\Delta \Bu$ in \eqref{model:NS}. Moreover, we set $\epsilon=1e-3$ in $W(t)$, and
use 128 Fourier modes for space discretization. 

We compute the mean of $300$ independent realizations at time points $t=0,0.4,0.8,1.2$, and present the evolution of vorticity contours in Fig.~\ref{fig2}. The top row corresponds to the deterministic case, where the shear layer gradually rolls up, forming spiral structures that evolve into a fully developed vortex. The bottom row shows the same initial configuration under additive noise $g(\Bu)=0.1$. Here, the noise significantly alters the dynamics: the rolling-up process is disrupted, and a vortex forms without the characteristic spirals, highlighting the intricate influence of stochastic perturbations.


\begin{figure*}[htbp]
	\begin{minipage}[t]{0.24\linewidth}
		\centerline{\includegraphics[scale=0.19]{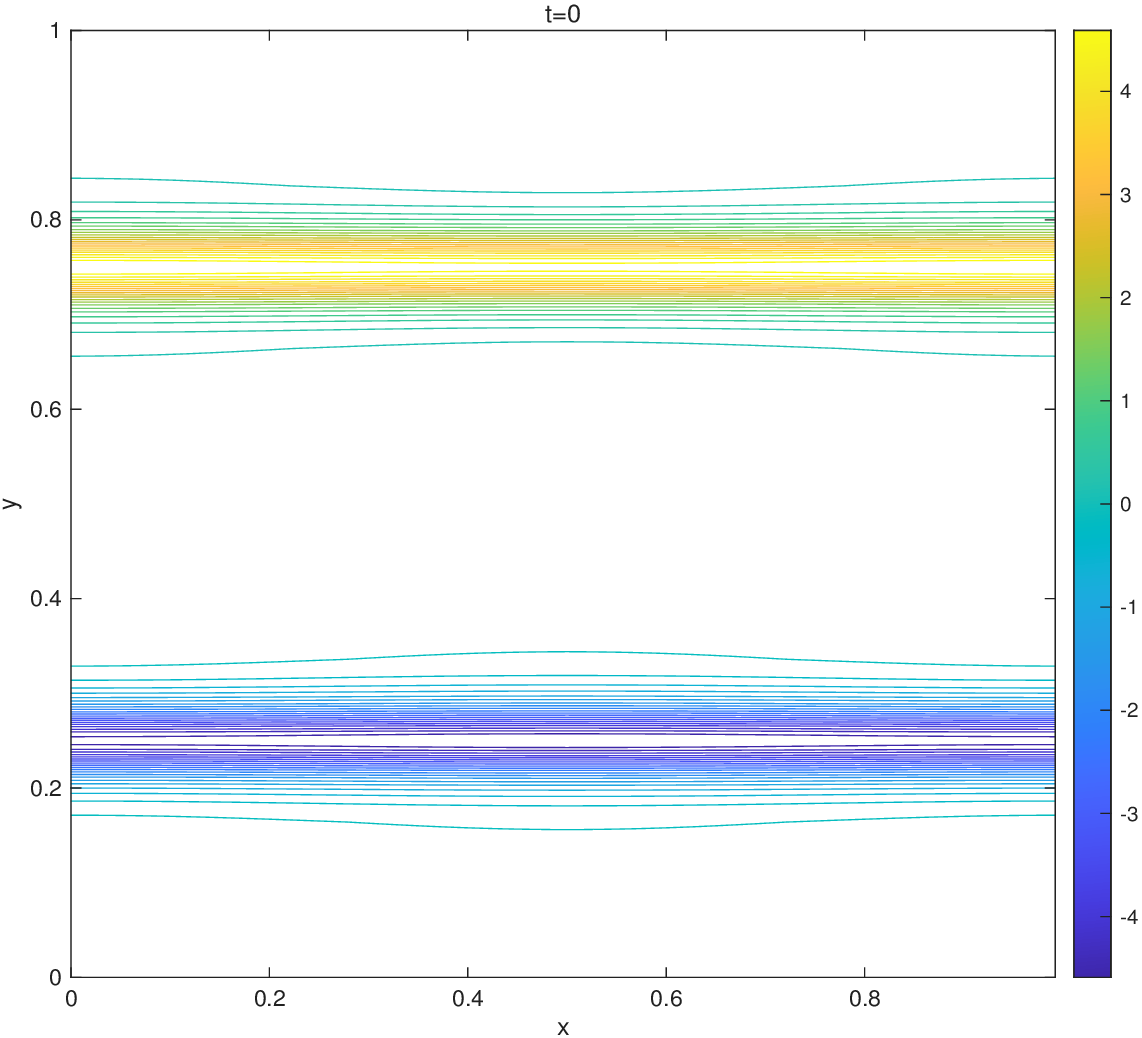}}
	\end{minipage}
	\begin{minipage}[t]{0.24\linewidth}
		\centerline{\includegraphics[scale=0.19]{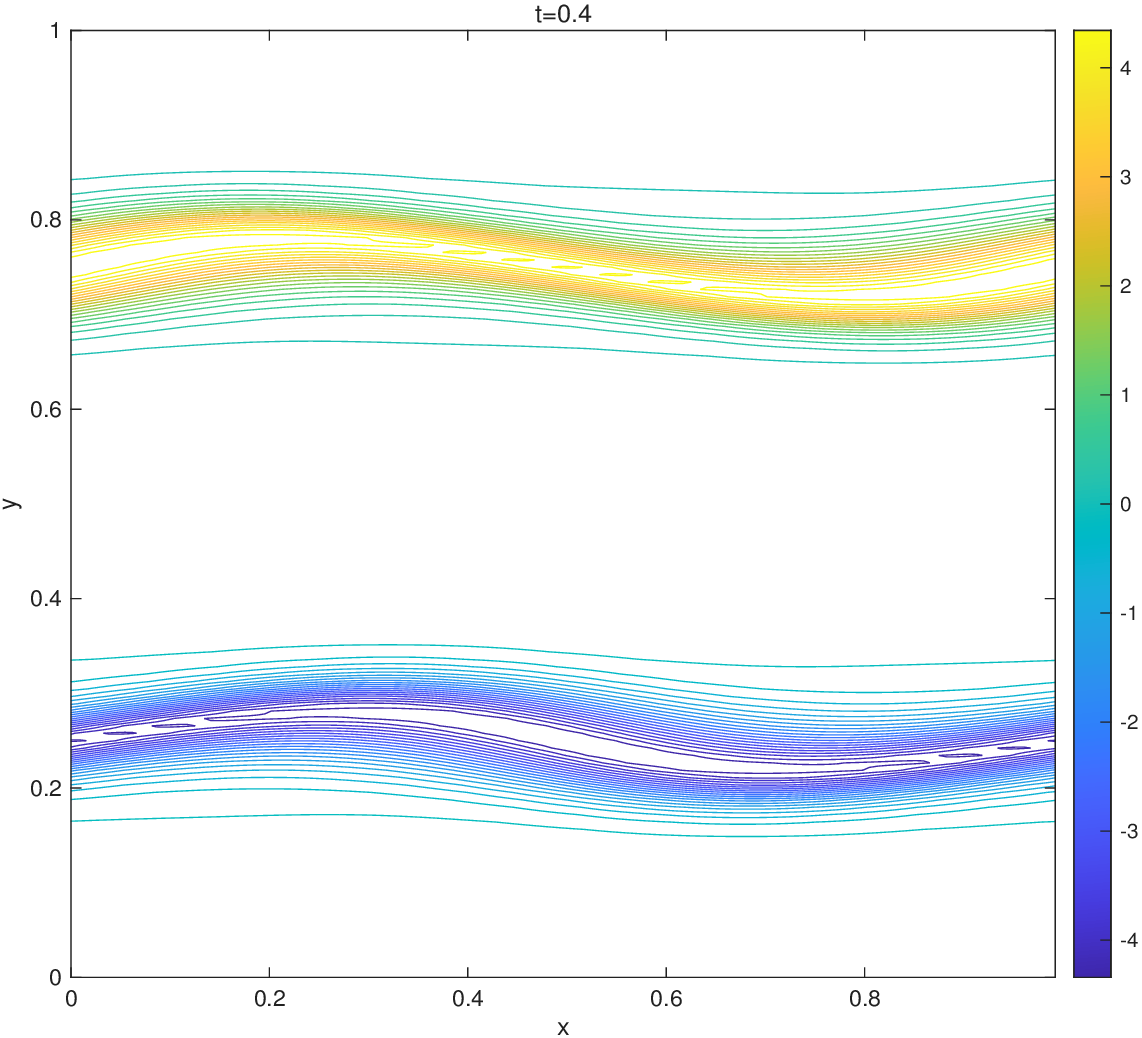}}
	\end{minipage}
	\begin{minipage}[t]{0.24\linewidth}
		\centerline{\includegraphics[scale=0.19]{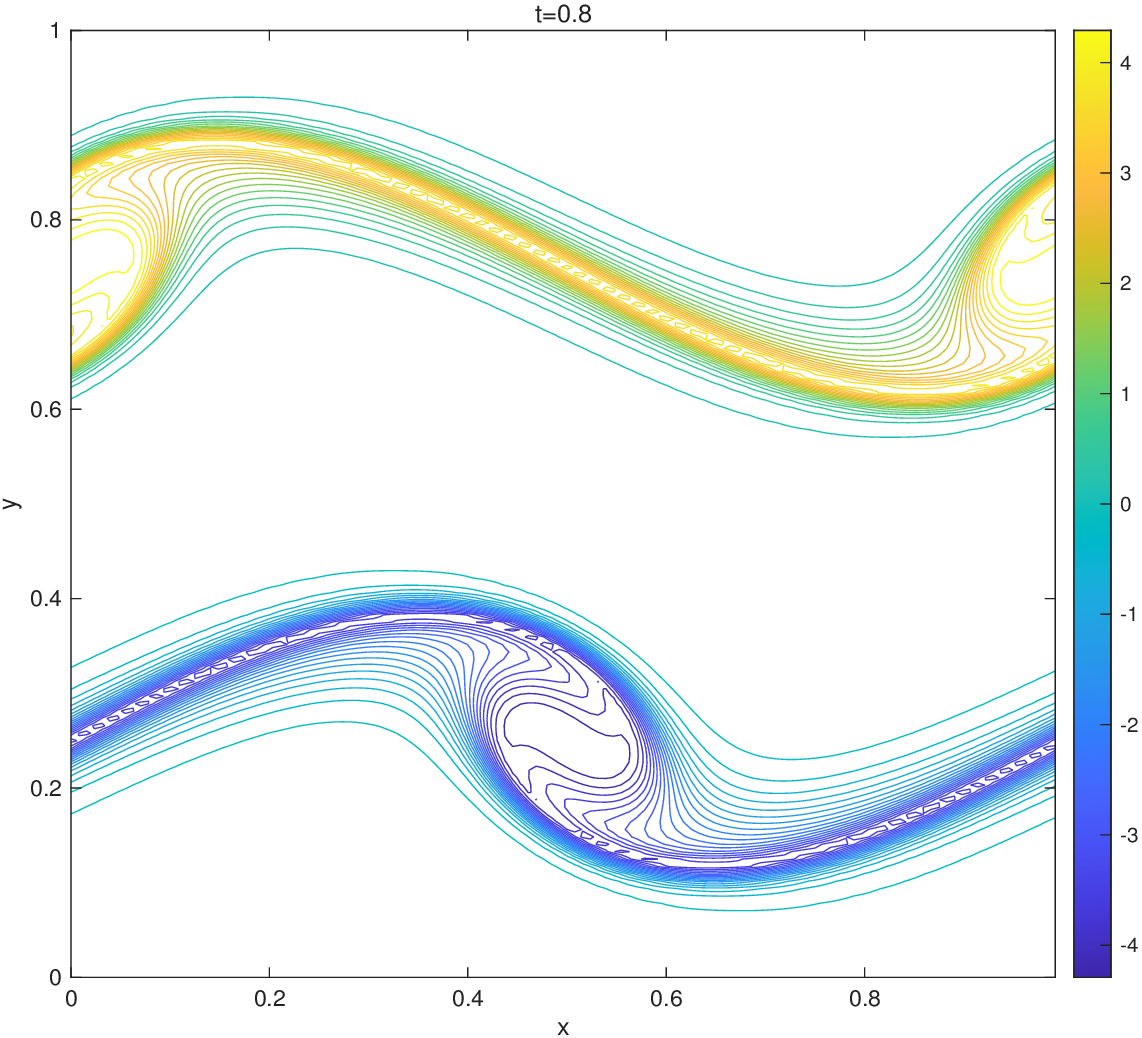}}
	\end{minipage}
	\begin{minipage}[t]{0.24\linewidth}
		\centerline{\includegraphics[scale=0.19]{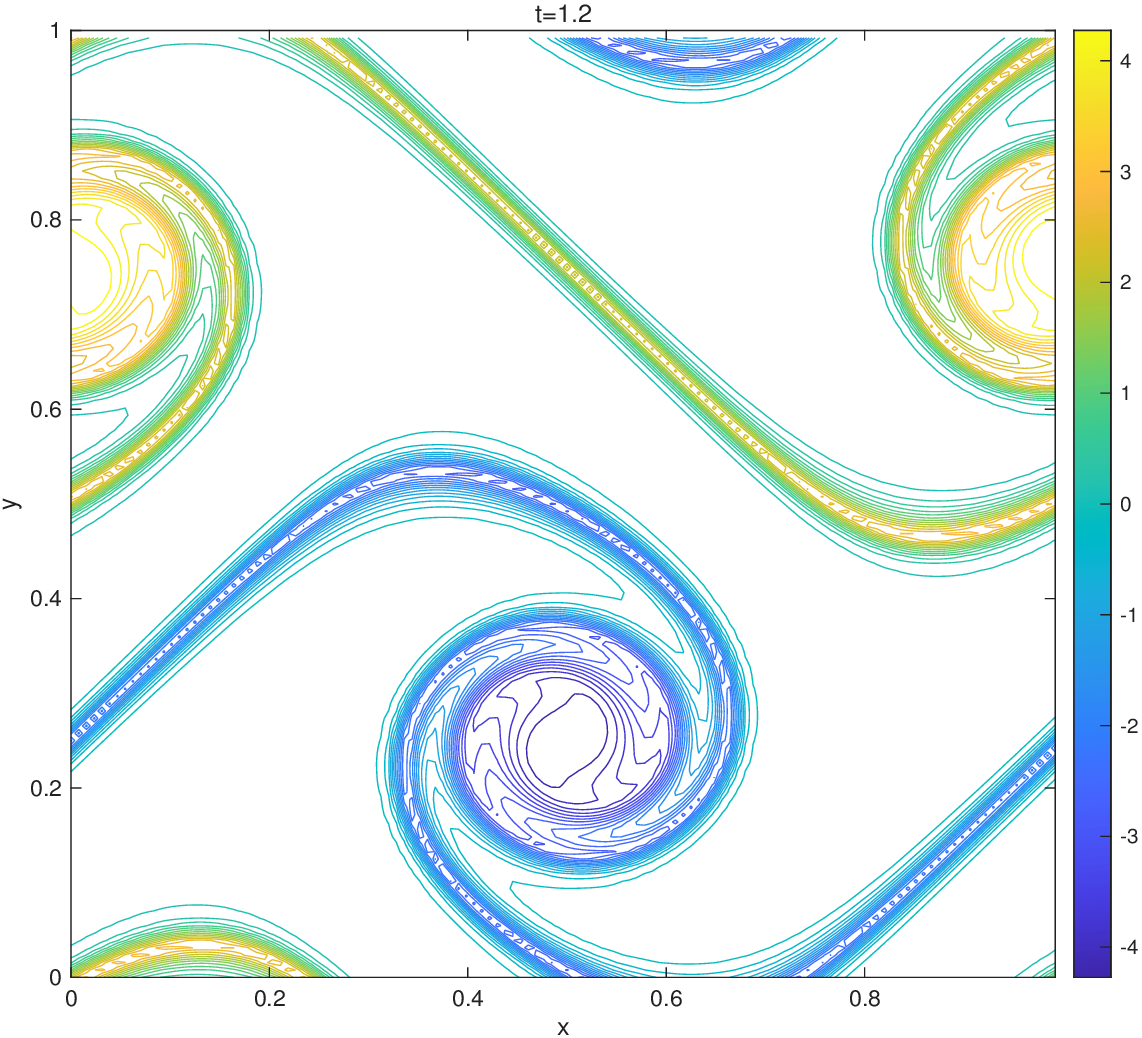}}
	\end{minipage}
	\vskip 3mm
	\begin{minipage}[t]{0.24\linewidth}
		\centerline{\includegraphics[scale=0.19]{vor_t0.eps}}
	\end{minipage}
	\begin{minipage}[t]{0.24\linewidth}
		\centerline{\includegraphics[scale=0.19]{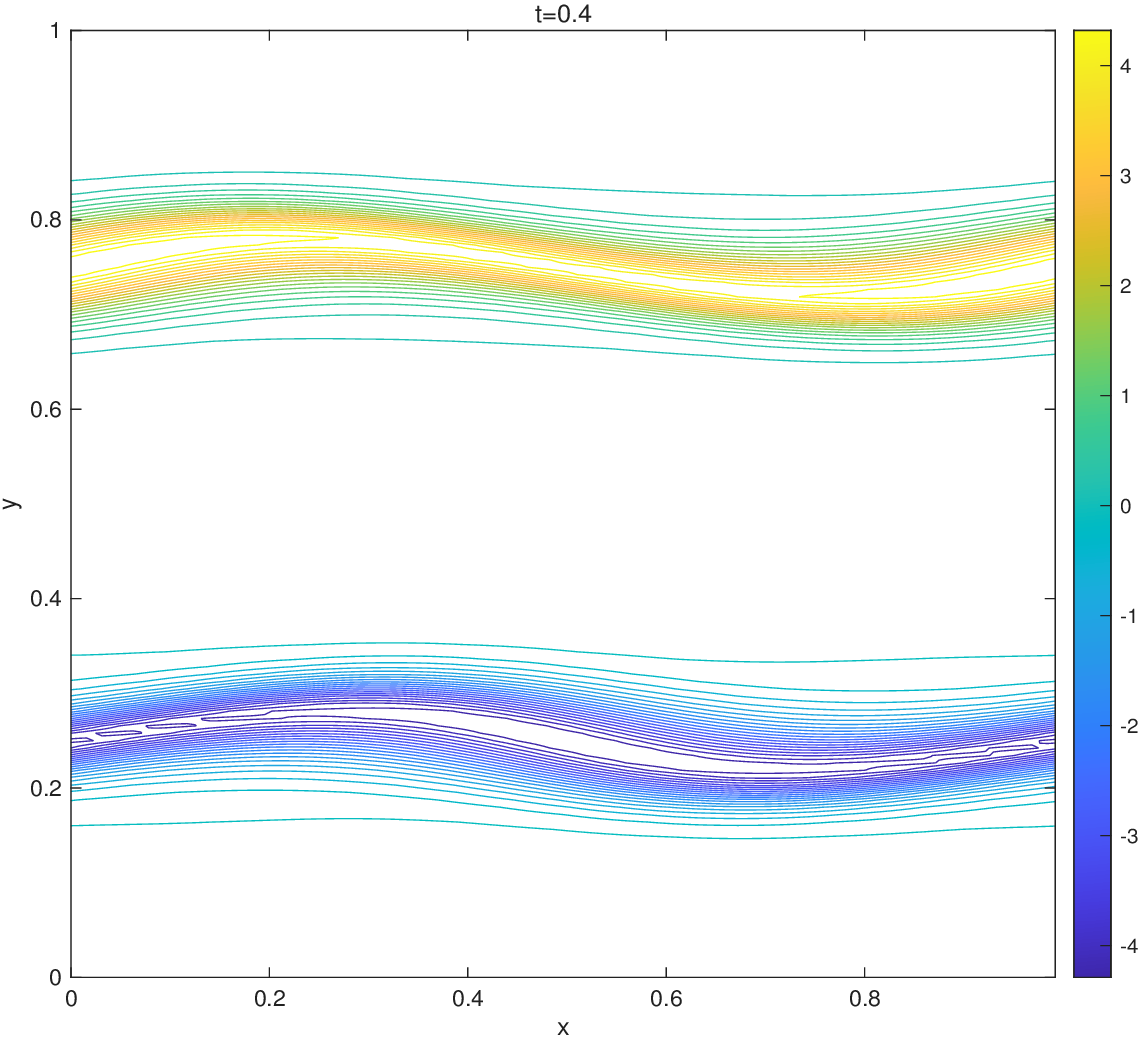}}
	\end{minipage}
	\begin{minipage}[t]{0.24\linewidth}
		\centerline{\includegraphics[scale=0.19]{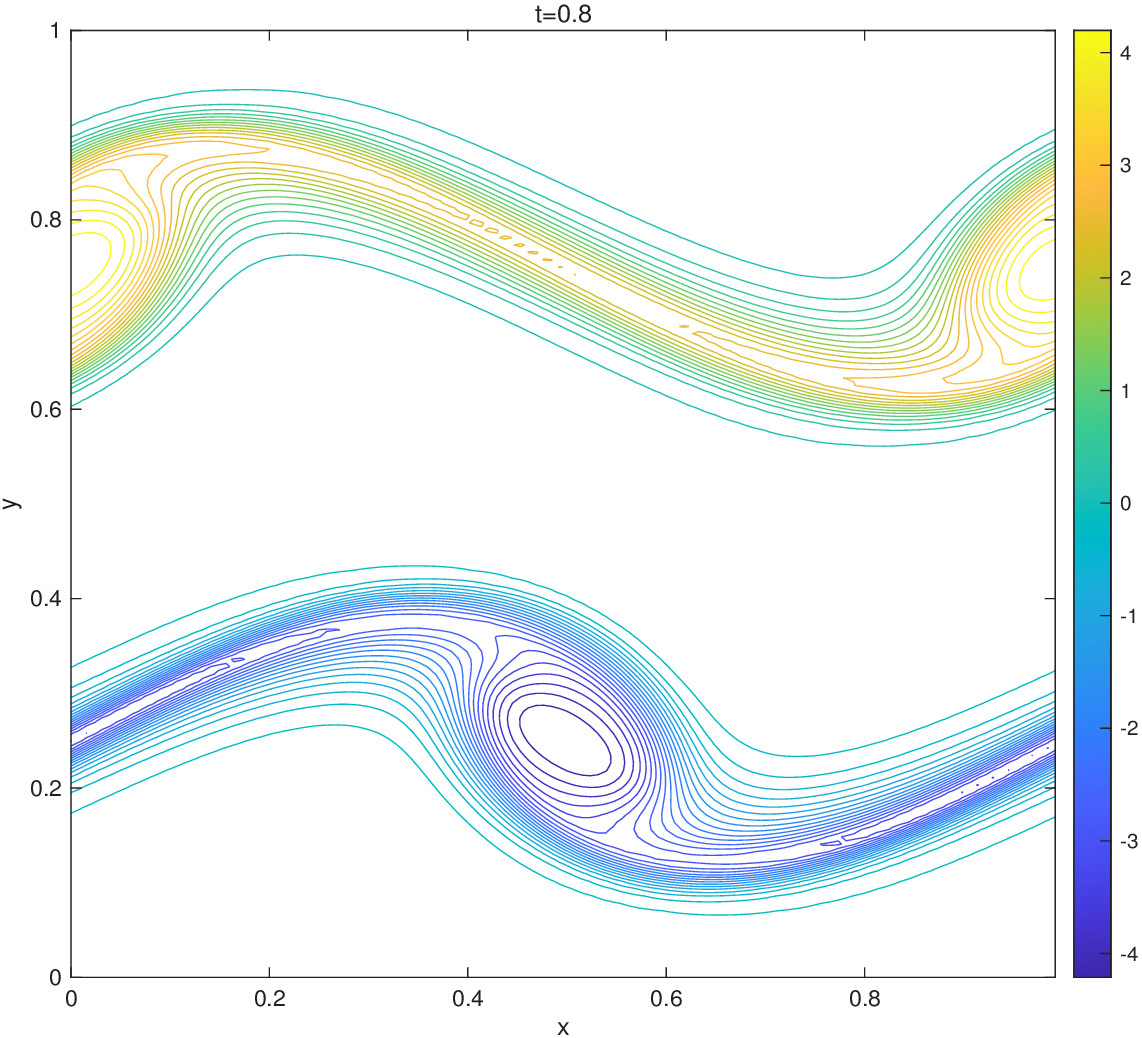}}
	\end{minipage}
	\begin{minipage}[t]{0.24\linewidth}
		\centerline{\includegraphics[scale=0.19]{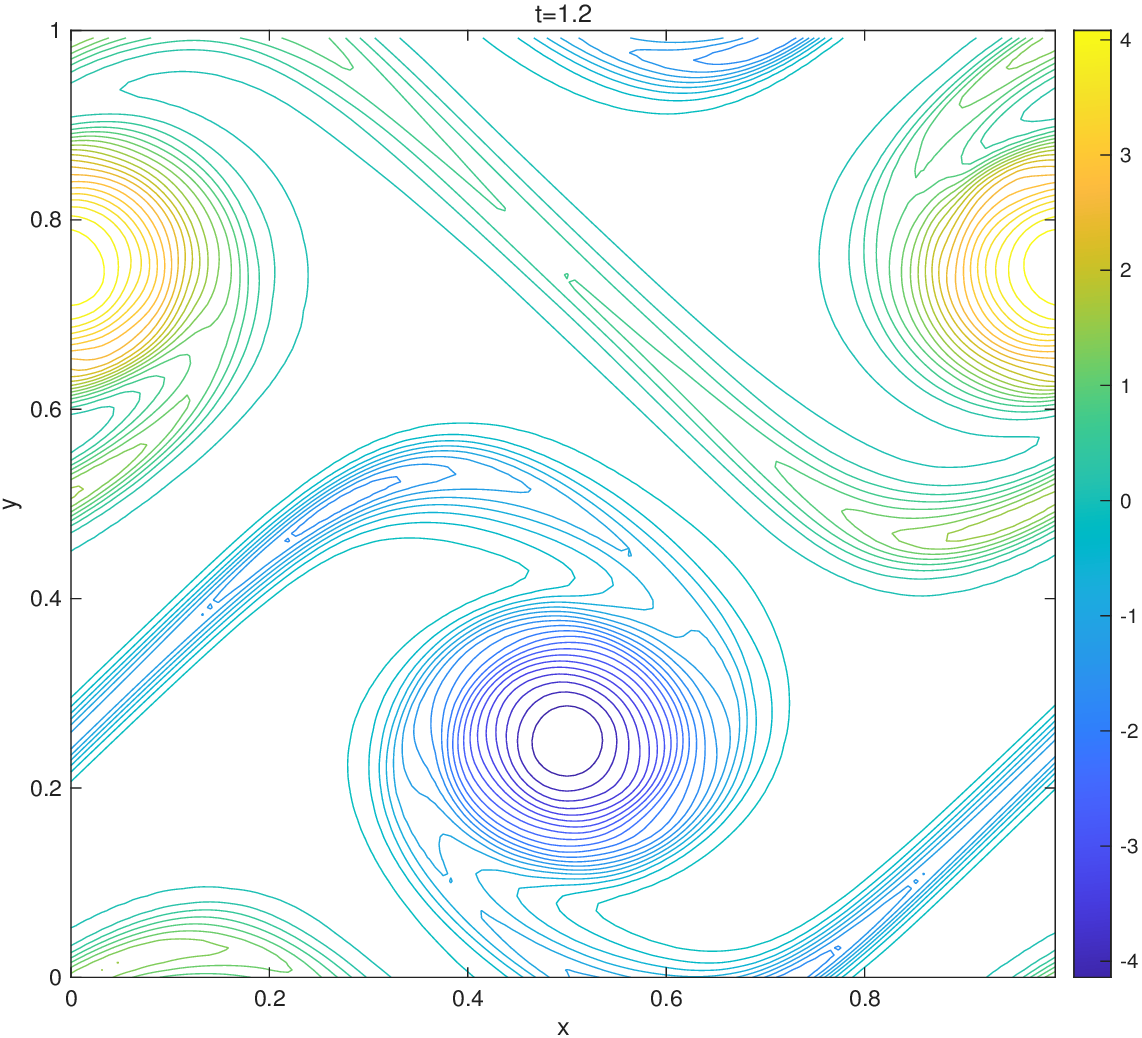}}
	\end{minipage}
	\caption{Time evolution of vorticity contours at $t=0, 0.4, 0.8,1.2$. The top row: deterministic case. The bottom row: stochastic case with $g(\Bu)={ 0.1}$.
	}\label{fig2}
\end{figure*}

\section{Concluding remarks}
We have proposed an efficient pressure-correction scheme based on auxiliary variables for the stochastic Navier-Stokes equations under Dirichlet boundary condition. This approach generalizes the periodic framework considered in \cite{CarelliP12}. The proposed scheme is linear, decoupled, and fully explicit, yet maintains unconditional stability, with an adaptive mean-reverting mechanism naturally embedded in the auxiliary variables. To the best of our knowledge, this is the first work on stochastic Navier-Stokes equations that treats the nonlinear term explicitly while preserving unconditional stability.

Our convergence analysis is currently confined to the linearized stochastic system \eqref{model:NS1}. This limitation arises because the quadrilinear term $\xi^{n+1}([\Bu^n\cdot\nabla]\Bu^n, \tilde{\Bu}^{n+1})$ introduced in the discretization of \eqref{model:NS} couples the auxiliary variable $\xi$ with the velocity $\Bu$, rendering the analysis substantially more challenging. By contrast, fully implicit or semi-linearized schemes involve only trilinear terms. Overcoming this difficulty will likely require localization techniques similar to those employed in \cite{CarelliP12}, and we leave this challenge for future investigation.

\bibliographystyle{plain}
\bibliography{Ref}

@Book{Batchelor92,
  author = {G.K. Batchelor},
  title = {An Introduction to Fluid Mechanics},
  year = {1992},
  publisher = {Cambridge University Press},
  address = {Cambridge}
}

@article {BessaihM22,
    AUTHOR = {H. Bessaih and A. Millet},
     TITLE = {Space-time {E}uler discretization schemes for the stochastic
              2{D} {N}avier-{S}tokes equations},
   JOURNAL = {Stoch. Partial Differ. Equ. Anal. Comput.},
    VOLUME = {10},
      YEAR = {2022},
    NUMBER = {4},
     PAGES = {1515--1558}
  }

@Article{BrehierCH19,
  author = {C. Brehier and J. Cui and J. Hong},
  title = {Strong convergence rates of semidiscrete splitting approximations for the stochastic {A}llen–{C}ahn equation},
  journal = {IMA J. Numer. Anal.},
  year = {2019},
  volume = {39},
  number={4},
  pages = {2096-2134}
}

@Book{BoyerF12,
  author = {F. Boyer and P. Fabrie},
  title = {Mathematical tools for the study of the incompressible {N}avier-{S}tokes equations and related models},
  year = {2012},
  publisher = {Springer},
  address = {New York}
}

@article {Breckner00,
    AUTHOR = {H. Breckner},
     TITLE = {Galerkin approximation and the strong solution of the {N}avier-{S}tokes equation},
   JOURNAL = {J. Appl. Math. Stochastic Anal.},
    VOLUME = {13},
      YEAR = {2000},
    NUMBER = {},
     PAGES = {239--259}
   }

@article {BrzezniakCP13,
    AUTHOR = {Z. Brze\'zniak and E. Carelli and A. Prohl},
     TITLE = {Finite-element-based discretizations of the incompressible
              {N}avier-{S}tokes equations with multiplicative random
              forcing},
   JOURNAL = {IMA J. Numer. Anal.},
    VOLUME = {33},
      YEAR = {2013},
    NUMBER = {3},
     PAGES = {771--824}
   }

@Article{BreitD21,
  author = {D. Breit and A. Dodgson},
  title = {Convergence rates for the numerical approximation of the
2{D} stochastic {N}avier–{S}tokes equations},
  journal = {Numer. Math.},
  year = {2021},
  volume = {147},
  number={},
  pages = {533-578}
}

@article {BreitD24,
    AUTHOR = {D. Breit and A. Dodgson},
     TITLE = {Space-time approximation of local strong solutions to the 3{D}
              stochastic {N}avier-{S}tokes equations},
   JOURNAL = {Comput. Methods Appl. Math.},
     VOLUME = {24},
      YEAR = {2024},
    NUMBER = {3},
     PAGES = {577--597}
  }

@article {BreitP23,
    AUTHOR = {D. Breit and A. Prohl},
     TITLE = {Numerical analysis of two-dimensional {N}avier–{S}tokes equations with additive
stochastic forcing},
   JOURNAL = {IMA J. Numer. Anal.},
     VOLUME = {43},
      YEAR = {2023},
    NUMBER = {},
     PAGES = {1391--1421}
  }

@article {BreitP24,
    AUTHOR = {D. Breit and A. Prohl},
     TITLE = {Error analysis for 2{D} Stochastic {N}avier–{S}tokes equations in bounded domains with Dirichlet data},
   JOURNAL = {Found. Comput. Math.},
     VOLUME = {24},
      YEAR = {2024},
    NUMBER = {},
     PAGES = {1643--1672}
  }

@Article{CarelliP12,
  author = {E. Carelli and A. Prohl},
  title = {Rates of convergence for discretizations of the stochastic incompressible {N}avier-{S}tokes equations},
  journal = {SIAM J. Numer. Anal.},
  year = {2012},
  volume = {50},
  number={5},
  pages = {2467--2496}
}

@Article{ColemanHW25,
  author = {J. Coleman and D. Han and X. Wang},
  title = {An efficient scheme for approximating long-time
dynamics of a class of non-linear models},
  journal = {Arxiv.org: 2411.03689v3},
  year = {2025},
  volume = {},
  pages = {}
}

@Article{CuiHL17,
  author = {J. Cui and J. Hong and Z. Liu},
  title = {Strong convergence rate of finite difference approximations for stochastic cubic Schr\"odinger equations},
  journal = {J. Diff. Equ},
  year = {2017},
  volume = {263},
  number={},
  pages = {3687--3713}
}

@book{DaPrato12,
author={G. Da Prato},
title={Kolmogorov equations for stochastic PDEs},
year={2012},
publisher={Birkh\"auser},
}

@inbook{Debussche13,
    author ={A. Debussche} ,
    title ={Ergodicity Results for the Stochastic {N}avier–{S}tokes Equations: An Introduction} ,
    publisher = {Springer} ,
    year ={2013} ,
    chapter ={In: Topics in Mathematical Fluid Mechanics. Lecture Notes in Mathematics}, 
}

@article {Dorsek12,
    AUTHOR = {P. D\"orsek},
     TITLE = {Semigroup splitting and cubature approximations for the
              stochastic {N}avier-{S}tokes equations},
   JOURNAL = {SIAM J. Numer. Anal.},
     VOLUME = {50},
      YEAR = {2012},
    NUMBER = {2},
     PAGES = {729--746}
    }

@article {Doghman24,
    AUTHOR = {J. Doghman},
     TITLE = {Numerical approximation of the stochastic {N}avier–{S}tokes equations through artificial compressibility},
   JOURNAL = {Calcolo},
     VOLUME = {23},
      YEAR = {2024},
    NUMBER = {},
     PAGES = {}
    }

@Article{FengV22,
  author = {X. Feng and L. Vo},
  title = {High moment and pathwise error estimates for fully discrete mixed finite element approximattions of stochastic {N}avier-{S}tokes equations with additive noise},
  journal = {Comm. Comput. Phys.},
  year = {2024},
  volume = {36},
  number={3},
  pages = {821--849},
}

@Article{Giga85,
  author = {Y. Giga and T. Miyakawa},
  title = {Solutions in $L_r$ of the {N}avier-{S}tokes initial value problem},
  journal = {Arch. Rational Mech. Anal.},
  year = {1985},
  volume = {89},
  number={},
  pages = {267--281}
}

@Article{Gourcy17,
  author = {M. Gourcy},
  title = {A large deviation principle for 2{D} stochastic
{N}avier–{S}tokes equation},
  journal = {Stoch. Proc. Appl.},
  year = {2017},
  volume = {117},
  number={},
  pages = {904--927}
}

@article {HairerM06,
    AUTHOR = {M. Hairer and J.C. mattingly},
     TITLE = {Ergodicity of the 2{D} {N}avier-{S}tokes equations with degenerate
stochastic forcing},
   JOURNAL = {Ann. of Math.},
     VOLUME = {164},
      YEAR = {2006},
    NUMBER = {2},
     PAGES = {993--1032}
}

@article {HausenblasR19,
    AUTHOR = {E. Hausenblas  and T. A. Randrianasolo},
     TITLE = {Time-discretization of stochastic 2{D} {N}avier-{S}tokes
              equations with a penalty-projection method},
   JOURNAL = {Numer. Math.},
     VOLUME = {143},
      YEAR = {2019},
    NUMBER = {2},
     PAGES = {339--378}
}

@Article{HutzenhalerJ11,
  author = {M. Hutzenthaler and A. Jentzen},
  title = {Convergence of the stochastic {E}uler Scheme for Locally Lipschitz Coefficients},
  journal = {Found. Comput. Math.},
  year = {2011},
  volume = {11},
  number={},
  pages = {657--706}
}

@book{KuksinS12,
author={S. Kuksin and A. Shirikyan},
title={Mathematics of two-dimensional turbulence},
year={2012},
publisher={Cambridge University Press}
}

@article {KukavicaUZ18,
    AUTHOR = {I. Kukavica and K. U\u gurlu  and M. Ziane},
     TITLE = {On the {G}alerkin approximation and strong norm bounds for the
              stochastic {N}avier-{S}tokes equations with multiplicative
              noise},
   JOURNAL = {Differential Integral Equations},
  FJOURNAL = {Differential and Integral Equations. An International Journal
              for Theory \& Applications},
    VOLUME = {31},
      YEAR = {2018},
    NUMBER = {3-4},
     PAGES = {173--186}
 }

@article {LiSL21,
    AUTHOR = {X.L. Li and J. Shen  and Z.G. Liu},
     TITLE = {New {SAV}-pressure correction methods for the
              {N}avier-{S}tokes equations: stability and error analysis},
   JOURNAL = {Math. Comp.},
     VOLUME = {91},
      YEAR = {2021},
    NUMBER = {333},
     PAGES = {141--167}
}

@article {LiXZ25,
    AUTHOR = {B.  Li and X. Xie and Q. Zhou},
     TITLE = {Pathwise uniform convergence of numerical
approximations for a two-dimensional stochastic
{N}avier-{S}tokes equation with no-slip boundary condition},
   JOURNAL = {arxiv:2412.04231v2},
     VOLUME = {},
      YEAR = {2025},
    NUMBER = {},
     PAGES = {}
}

@article {LiMS25,
    AUTHOR = {B.  Li and S. Ma and W. Sun},
     TITLE = {Optimal analysis of finite element methods for the stochastic {S}tokes equations},
   JOURNAL = {Math. Comput.},
     VOLUME = {94},
      YEAR = {2025},
    NUMBER = {352},
     PAGES = {551--583}
}

@Article{LiuQ21,
  author = {Z. Liu and Z. Qiao},
  title = {Strong approximation of monotone stochastic partial differential equations driven by multiplicative noise},
  journal = {Stoch PDE: Anal. Comp.},
  year = {2021},
  volume = {9},
  pages = {559--602}
}

@Article{LLS00,
  author = {C.O. Lee and J. Lee and D. Sheen},
  title = {Frequency domain formulation of linearized {N}avier–{S}tokes
equations},
  journal = {Comput. Methods Appl. Mech. Engrg. },
  year = {2000},
  volume = {187},
  pages = {351--362}
}

@Article{KovacsLL11,
  author = {M. Kov\'acs and  F. Lindgren and S. Larsson},
  title = {Spatial approximation of stochastic convolutions},
  journal = {J. Comput. Appl. Math.},
  year = {2011},
  volume = {235},
  number={},
  pages = {3554--3570}
}

@Article{KovacsLL15,
  author = {M. Kov\'acs and S. Larsson and F. Lindgren},
  title = {On the backward {E}uler approximation of the stochastic {A}llen-{C}ahn equation},
  journal = {J. Appl. Probab.},
  year = {2015},
  volume = {52},
  number={2},
  pages = {323--338}
}

@book{Kruse14,
author={R. Kruse},
title={Strong and weak approximation of semilinear stochastic evolution equations},
year={2014},
publisher={Springer},
}

@book{LordPS14,
author={G. Lord and C. Powell and T. Shardlow},
title={An introduction to computational stochastic {PDE}s},
year={2014},
publisher={Cambridge university press},
}

@article {OndrejatPW23,
    AUTHOR = {M. Ondrej\'at and A. Prohl and N. Walkington},
     TITLE = {Numerical approximation of nonlinear {SPDE}’s},
   JOURNAL = {Stoch PDE: Anal. Comp.},
    VOLUME = {11},
      YEAR = {2023},
    NUMBER = {},
     PAGES = {1553--1634},
}

@book{Oseen27,
    AUTHOR = {C.W. Oseen},
     TITLE = {Neuere Methoden und Ergebnisse in der Hydrodynamik},
    SERIES = {Akademische Verlagsgesellschaft},
      YEAR = {1927},
     publisher = {M.B.H. Leipzig}
}

@article {QiW20,
    AUTHOR = {R. Qi and X. Wang},
     TITLE = {Error estimates of semidiscrete and fully discrete finite
              element methods for the {C}ahn-{H}illiard-{C}ook equation},
   JOURNAL = {SIAM J. Numer. Anal.},
    VOLUME = {58},
      YEAR = {2020},
    NUMBER = {3},
     PAGES = {1613--1653},
}

\end{document}